\documentclass[oneside,english]{amsart}
\usepackage[T1]{fontenc}
\usepackage[latin9]{inputenc}
\usepackage{amstext}
\usepackage{amsthm}
\usepackage{amssymb}

\makeatletter

\providecommand{\tabularnewline}{\\}

\theoremstyle{plain}
\newtheorem{thm}{\protect\theoremname}
\theoremstyle{definition}
\newtheorem{defn}[thm]{\protect\definitionname}
\theoremstyle{plain}
\newtheorem{lem}[thm]{\protect\lemmaname}
\theoremstyle{definition}
\newtheorem{example}[thm]{\protect\examplename}
\theoremstyle{plain}
\newtheorem{cor}[thm]{\protect\corollaryname}
\theoremstyle{plain}
\newtheorem{prop}[thm]{\protect\propositionname}

\usepackage{xypic}
\usepackage{graphicx}
\usepackage{enumitem}
\usepackage{mathrsfs, setspace, fancyhdr, times, bm}
\DeclareFontFamily{OT1}{pzc}{}
\DeclareFontShape{OT1}{pzc}{m}{it}%
             {<-> s * [1.195] pzcmi7t}{}
\DeclareMathAlphabet{\mathscr}{OT1}{pzc}%
                                 {m}{it}
\usepackage[colorlinks=true,pagebackref=true]{hyperref} % new hyperref package added by Jean
\hypersetup{backref}

\makeatother

\usepackage{babel}
\providecommand{\corollaryname}{Corollary}
\providecommand{\definitionname}{Definition}
\providecommand{\examplename}{Example}
\providecommand{\lemmaname}{Lemma}
\providecommand{\propositionname}{Proposition}
\providecommand{\theoremname}{Theorem}

\begin{document}


\thanks{This work received support from the French \textquotedbl Investissements
d\textquoteright Avenir\textquotedbl{} program, project ISITE-BFC
(contract ANR-lS-IDEX-OOOB) and from ANR Project FIBALGA ANR-18-CE40-0003-01. }

\author{Adrien Dubouloz}

\address{IMB UMR5584, CNRS, Univ. Bourgogne Franche-Comté, F-21000 Dijon,
France.}

\email{adrien.dubouloz@u-bourgogne.fr}

\author{Charlie Petitjean }

\address{I.U.T. Dijon-Département GMP, Boulevard Dr. Petitjean, 21078 Dijon,
France. }

\email{charlie.petitjean@u-bourgogne.fr}

\dedicatory{Dedicated to Lucy Moser-Jauslin on her 60th birthday}

\title{Rational real algebraic models of compact differential surfaces with
circle actions}
\begin{abstract}
We give an algebro-geometric classification of smooth real affine
algebraic surfaces endowed with an effective action of the real algebraic
circle group $\mathbb{S}^{1}$ up to equivariant isomorphisms. As
an application, we show that every compact differentiable surface
endowed with an action of the circle $S^{1}$ admits a unique smooth
rational real quasi-projective model up to $\mathbb{S}^{1}$-equivariant
birational diffeomorphism. 
\end{abstract}

\maketitle

\section*{Introduction}

A description of normal complex affine surfaces admitting an effective
action of the complex multiplicative group $\mathbb{G}_{m,\mathbb{C}}=(\mathbb{C}^{*},\times)$
was given by Flenner and Zaidenberg \cite{FZ03} in terms of their
graded coordinate rings. Generalizing earlier constructions due to
Dolgachev-Pinkham-Demazure \cite{Do75,Pi77,De88}, they described
these graded rings as rings of sections of divisors with rational
coefficients on suitable one-dimensional rational quotients of the
given action. This type of presentation, which is nowadays called
the DPD-\emph{presentation} of normal affine surfaces with $\mathbb{G}_{m,\mathbb{C}}$-actions,
was generalized vastly by Altmann and Hausen \cite{AH} to give presentations
of normal complex affine varieties of any dimension endowed with effective
actions of tori $\mathbb{G}_{m,\mathbb{C}}^{r}$ in terms of so-called
polyhedral Weil divisors on suitable rational quotients obtained as
limits of GIT quotients.

For normal affine varieties over arbitrary base fields $k$ endowed
with effective actions of non necessarily split tori, that is, commutative
$k$-groups schemes $G$ whose base extension to a separable closure
$k^{s}$ of $k$ are isomorphic to split tori $\mathbb{G}_{m,,k^{s}}^{r}$,
only partial extensions of the Altmann-Hausen formalism are known
so far. Besides the toric case considered by several authors, a first
step was made by Langlois \cite{La11} who obtained a generalization
to affine varieties endowed with effective actions of complexity one.
In another direction, Liendo and the first author \cite{DuLi18} recently
extended the Altmann-Hausen framework to describe normal real affine
varieties endowed with an effective action of the $1$-dimensional
non-split real torus, the circle 
\[
\mathbb{S}^{1}=\mathrm{Spec}(\mathbb{R}[u,v]/(u^{2}+v^{2}-1))\cong\mathrm{SO}_{2}(\mathbb{R}).
\]
The common approach in these generalizations is based on the understanding
of the interplay between the algebro-combinatorial data in an Altmann-Hausen
presentation of the variety with split torus action obtained by base
extension to a separable closure $k^{s}$ of the base field $k$ and
Galois descent with respect to the Galois group $\mathrm{Gal}(k^{s}/k)$.
In the real case, this amounts to describe normal real affine varieties
as normal complex affine varieties endowed with an anti-regular involution
$\sigma$, called a \emph{real structure. }The results in \cite{DuLi18}
then essentially consist of a description of $\mathbb{S}^{1}$-actions
on normal real affine varieties in the language of \cite{AH} extended
to complex affine varieties with real structures. 

The goal of this article is to give a survey of this description and
some applications in the special case of normal real affine surfaces,
formulated in the language of DPD-presentations of Flenner and Zaidenberg.
The first main result, Theorem \ref{thm:MainThm}, describes the one-to-one
correspondence between normal real affine surfaces with effective
$\mathbb{S}^{1}$-actions and certain pairs $(D,h)$ called \emph{real}
DPD-\emph{pairs} on smooth real affine curves $C$, consisting of
a Weil $\mathbb{Q}$-divisor $D$ on the complexification of $C$
and rational function $h$ on $C$. We also characterize which such
pairs correspond to smooth real affine surfaces. A second main result,
Theorem \ref{prop:Real-Except-Orb-1}, consists of a classification
of real $\mathbb{S}^{1}$-orbits on a smooth real affine surface in
relation to the structure of the real fibers of the quotient morphism
for the $\mathbb{S}^{1}$-action. 

To give an illustration of the flavor of these results, consider the
smooth complex affine surface 
\[
S=\{xy^{2}=1-z^{2}\}\subset\mathrm{Spec}(\mathbb{C}[x^{\pm1},y,z]).
\]
The group $\mathbb{G}_{m,\mathbb{C}}$ acts effectively on $S$ by
$t\cdot(x,y,z)=(t^{2}x,t^{-1}y,z)$ and the categorical quotient for
this action is the projection $\pi=\mathrm{pr}_{z}:S\rightarrow\mathbb{A}_{\mathbb{C}}^{1}=\mathrm{Spec}(\mathbb{C}[z])$.
All fibers of $\pi$ consist of a unique $\mathbb{G}_{m,\mathbb{C}}$-orbit,
isomorphic to $\mathbb{G}_{m,\mathbb{C}}$ acting on itself by translations,
except for $\pi^{-1}(-1)$ and $\pi^{-1}(1)$ which are isomorphic
to $\mathbb{G}_{m,\mathbb{C}}$ on which $\mathbb{G}_{m,\mathbb{C}}$
acts with stabilizer equal to the group $\mu_{2}$ of complex square
roots of unity. The composition of the involution $t\mapsto t^{-1}$
of $\mathbb{G}_{m,\mathbb{C}}$ with the complex conjugation defines
a real structure $\sigma_{0}$ on $\mathbb{G}_{m,\mathbb{C}}$ for
which the pair $(\mathbb{G}_{m,\mathbb{C}},\sigma_{0})$ describes
a real algebraic group isomorphic to $\mathbb{S}^{1}$. The composition
of the involution $(x,y,z)\mapsto(x^{-1},xy,z)$ of $S$ with the
complex conjugation defines a real structure $\sigma$ on $S$, making
the pair $(S,\sigma)$ into a smooth real affine surface. The $\mathbb{G}_{m,\mathbb{C}}$-action
on $S$ is compatible with these two real structures and defines a
real action of $\mathbb{S}^{1}$ on the real affine surface $(S,\sigma)$.
The quotient morphism $\pi$ can in turn be interpreted as a real
morphism $\pi:(S,\sigma)\rightarrow\mathbb{A}_{\mathbb{R}}^{1}=\mathrm{Spec}(\mathbb{R}[z])$
which is the categorical quotient of $(S,\sigma)$ for the $\mathbb{S}^{1}$-action
in the category of real algebraic varieties. A real DPD-pair $(D,h)$
describing the $\mathbb{S}^{1}$-action on $(S,\sigma)$ is then given
for instance by the Weil $\mathbb{Q}$-divisor $D=\frac{1}{2}\{-1\}+\frac{1}{2}\{1\}$
on $\mathbb{A}_{\mathbb{C}}^{1}$ and the rational function $h=1-z^{2}$
on $\mathbb{A}_{\mathbb{R}}^{1}$. The fibers of $\pi:(S,\sigma)\rightarrow\mathbb{A}_{\mathbb{R}}^{1}$
over real points $c$ of $\mathbb{A}_{\mathbb{R}}^{1}$ all consist
of a single $\mathbb{S}^{1}$-orbit which is isomorphic to $\mathbb{S}^{1}$-acting
on itself by translations if $c\in]-1,1[$, to $\mathbb{S}^{1}$ acting
on itself with stabilizer $\mu_{2}$ if $c\in\{-1,1\}$, and to the
real affine curve without real point $\{u^{2}+v^{2}=-1\}$ otherwise.
\\

The set of real points of the above surface $(S,\sigma)$ endowed
with the induced Euclidean topology is diffeomorphic to the Klein
bottle $K$ (see subsection \ref{subsec:KleinB-model}). Furthermore,
the $\mathbb{S}^{1}$-action on $(S,\sigma)$ induces a differentiable
action of the real circle $S^{1}$ on $K$ which coincides with the
standard $S^{1}$-action on $K$ viewed as the $S^{1}$-equivariant
connected sum $\mathbb{RP}^{2}\sharp_{S^{1}}\mathbb{RP}^{2}$ of two
copies of the projective plane $\mathbb{RP}^{2}$. In other words,
$(S,\sigma)$ endowed with its $\mathbb{S}^{1}$-action is an equivariant
\emph{real affine algebraic model} of the Klein bottle $K$ endowed
with its differentiable $S^{1}$-action. By a result of Comessatti
\cite{Co14}, a compact connected differential manifold of dimension
2 without boundary admits a projective rational real algebraic model
if and only if it is non-orientable or diffeomorphic to the sphere
$S^{2}$ or the torus $T=S^{1}\times S^{1}$. It was established later
on by Biswas and Huisman \cite{BiHu07} that such a projective model
is unique up to so-called \emph{birational diffeomorphism}s, that
is, diffeomorphisms induced by birational maps defined at every real
point and admitting inverses of the same type. As an application of
the real DPD-presentation formalism, we establish the following uniqueness
property of rational models of compact differentiable surfaces with
$S^{1}$ -actions among all smooth rational quasi-projective real
algebraic surfaces with $\mathbb{S}^{1}$-actions. 
\begin{thm}
\label{thm:MainThm-Model} A connected compact real differential manifold
of dimension $2$ without boundary endowed with an effective differentiable
$S^{1}$-action admits a smooth rational quasi-projective real algebraic
model with $\mathbb{S}^{1}$-action, unique up to $\mathbb{S}^{1}$-equivariant
birational diffeomorphism.
\end{thm}

The article is organized as follows. In the first section we review
the equivalence of categories between quasi-projective real varieties
and quasi-projective complex varieties equipped with real structures
and recall the interpretation of $\mathbb{S}^{1}$-actions on such
varieties as forms of $\mathbb{G}_{m,\mathbb{C}}$-actions on complex
varieties with real structures. We also describe a correspondence
between $\mathbb{S}^{1}$-torsors and certain pairs consisting of
a invertible sheaf and a rational function on their base space which
to our knowledge did not appear before in the literature. The second
section is devoted to the presentation of smooth real affine surfaces
with $\mathbb{S}^{1}$-action in terms of real DPD-pairs and to the
study of their real $\mathbb{S}^{1}$-orbits. In the third section,
we first present different constructions of rational projective and
affine real algebraic models of compact differential surfaces with
$S^{1}$-actions and then proceed to the proof of Theorem \ref{thm:MainThm-Model}. 

\newpage

\section{Preliminaries}

In what follows the term algebraic variety over field $k$ refers
to a geometrically integral scheme $X$ of finite type over $k$.
In the sequel, $k$ will be equal to either $\mathbb{R}$ or $\mathbb{C}$,
and we will say that $X$ is a real, respectively complex, algebraic
variety. 

\subsection{Real and complex quasi-projective algebraic varieties}

Every complex algebraic variety $V$ can be viewed as a scheme over
the field $\mathbb{R}$ of real numbers via the composition of its
structure morphism $p:V\rightarrow\mathrm{Spec}(\mathbb{C})$ with
the \'etale double cover $\mathrm{Spec}(\mathbb{C})\rightarrow\mathrm{Spec}(\mathbb{R})$
induced by the inclusion $\mathbb{R}\hookrightarrow\mathbb{C}=\mathbb{R}[i]/(i^{2}+1)$.
The Galois group $\mathrm{Gal}(\mathbb{C}/\mathbb{R})=\mu_{2}$ acts
on $\mathrm{Spec}(\mathbb{C})$ by the usual complex conjugation $z\mapsto\overline{z}$. 
\begin{defn}
A \emph{real structure} on a complex algebraic variety $V$ consists
of an involution of $\mathbb{R}$-scheme $\sigma$ of $V$ which lifts
the complex conjugation, so that we have a commutative diagram \[\xymatrix{V \ar[rr]^{\sigma} \ar[d]_{p} && V \ar[d]^{p} \\ \mathrm{Spec}(\mathbb{C}) \ar[rr]^{z\mapsto\overline{z}} \ar[dr] &&  \mathrm{Spec}(\mathbb{C}) \ar[dl] \\ & \mathrm{Spec}(\mathbb{R}). }\] 
When $V=\mathrm{Spec}(A)$ is affine, a real structure $\sigma$ is
equivalently determined by its comomorphism $\sigma^{*}:A\rightarrow A$,
which is an involution of $A$ viewed as an $\mathbb{R}$-algebra. 

A \emph{real morphism (resp. real rational map)} between complex algebraic
varieties with real structures $(V',\sigma')$ and $(V,\sigma)$ is
a morphism (resp. a rational map) of complex algebraic varieties $f:V'\rightarrow V$
such that $\sigma\circ f=f\circ\sigma'$ as morphisms of $\mathbb{R}$-schemes. 
\end{defn}

For every real algebraic variety $X$, the \emph{complexification}
\[
X_{\mathbb{C}}=X\times_{\mathbb{R}}\mathbb{C}:=X\times_{\mathrm{Spec}(\mathbb{R})}\mathrm{Spec}(\mathbb{C})
\]
of $X$ comes equipped with a canonical real structure $\sigma_{X}$
given by the action of $\mathrm{Gal}(\mathbb{C}/\mathbb{R})$ by complex
conjugation on the second factor. Conversely, if a complex variety
$p:V\rightarrow\mathrm{Spec}(\mathbb{C})$ is equipped with a real
structure $\sigma$ and covered by $\sigma$-invariant affine open
subsets -so for instance if $V$ is quasi-projective-, then the quotient
$\pi:V\rightarrow V/\langle\sigma\rangle$ exists in the category
of schemes and the structure morphism $p:V\rightarrow\mathrm{Spec}(\mathbb{C})$
descends to a morphism $V/\langle\sigma\rangle\rightarrow\mathrm{Spec}(\mathbb{R})=\mathrm{Spec}(\mathbb{C})/\langle z\mapsto\overline{z}\rangle$
making $V/\langle\sigma\rangle$ into a real algebraic variety $X$
such that $V\cong X_{\mathbb{C}}$. In the case where $V=\mathrm{Spec}(A)$
is affine, the algebraic variety $X=V/\langle\sigma\rangle$ is affine,
equal to the spectrum of the ring $A^{\sigma^{*}}$ of $\sigma^{*}$-invariant
elements of $A$. This correspondence extends to a well-known equivalence
of categories:
\begin{lem}
\label{lem:Real-structures}\cite{BS64} The functor $X\mapsto(X_{\mathbb{C}},\sigma_{X})$
is an equivalence between the category of quasi-projective real algebraic
varieties and the category of pairs $(V,\sigma)$ consisting of a
complex quasi-projective variety $V$ endowed with a real structure
$\sigma$. 
\end{lem}

In what follows, we will switch freely from one point of view to the
other, by viewing a quasi-projective real algebraic variety $X$ either
as a geometrically integral $\mathbb{R}$-scheme of finite type or
as a pair $(V,\sigma)$ consisting of a quasi-projective complex algebraic
variety $V$ endowed with a real structure $\sigma$. A \emph{real
form} of a given real algebraic variety $(V,\sigma)$ is a real algebraic
variety $(V',\sigma')$ such that the complex varieties $V$ and $V'$
are isomorphic. A \emph{real closed subscheme} $Z$ of a real algebraic
variety $(V,\sigma)$ is a $\sigma$-invariant closed subscheme of
$V$, endowed with the induced real structure $\sigma|_{Z}$. \\

The set $V(\mathbb{C})$ of complex points of a smooth complex algebraic
variety $V$ can be endowed with a natural structure of real smooth
manifold locally inherited from that on $\mathbb{A}_{\mathbb{C}}^{n}(\mathbb{C})\simeq\mathbb{C}^{n}\simeq\mathbb{R}^{2n}$
\cite[Lemme 1 and Proposition 2]{Se55}. Every morphism of smooth
complex algebraic varieties $f:V'\rightarrow V$ induces a differentiable
map $f(\mathbb{C}):V(\mathbb{C})\rightarrow V(\mathbb{C})$ which
is a diffeomorphism when $f$ is an isomorphism. Similarly, a real
structure $\sigma$ on $V$ induces a differentiable involution of
$V(\mathbb{C})$, whose set of fixed points $V(\mathbb{C})^{\sigma}$,
called the \emph{real locus} of $(V,\sigma)$, is a smooth differential
real manifold. The real algebraic variety $(V,\sigma)$ is then said
to be an \emph{algebraic model} of this differential manifold.
\begin{defn}
A\emph{ birational diffeomorphism} $\varphi:(V',\sigma')\dashrightarrow(V,\sigma)$
between smooth real algebraic varieties with non empty real loci is
a real birational map whose restriction to the real locus $V'(\mathbb{C})^{\sigma'}$
of $(V',\sigma')$ is a diffeomorphism onto the real locus $V(\mathbb{C})^{\sigma}$
of $(V,\sigma)$, and which admits a rational inverse of the same
type. 
\end{defn}

\begin{example}
\label{exa:StereographicProj}Let $(Q_{1,\mathbb{C}},\sigma_{Q_{1}})$
be the complexification of the smooth affine curve $Q_{1}$ in $\mathbb{A}_{\mathbb{R}}^{2}=\mathrm{Spec}(\mathbb{R}[u,v])$
defined by the equation $u^{2}+v^{2}=1$. The stereographic projection
from the real point $N=(0,1)$ of $(Q_{1,\mathbb{C}},\sigma_{Q_{1}})$
induces an everywhere defined birational diffeomorphism 
\[
\pi_{N}:(Q_{1,\mathbb{C}}\setminus\{N\},\sigma_{Q_{1}}|_{Q_{1,\mathbb{C}}\setminus\{N\}})\rightarrow(\mathbb{A}_{\mathbb{C}}^{1}=\mathrm{Spec}(\mathbb{C}[z]),\sigma_{\mathbb{A}_{\mathbb{R}}^{1}}),\quad(u,v)\mapsto\frac{u}{1-v}
\]
with image equal to $\mathbb{A}_{\mathbb{C}}^{1}\setminus\{\pm i\}$.
Its inverse is given by 
\[
z\mapsto(u,v)=(\frac{2z}{z^{2}+1},\frac{z^{2}-1}{z^{2}+1}).
\]
\end{example}

\subsection{Circle actions as real forms of hyperbolic $\mathbb{G}_{m}$-actions}
\begin{defn}
\label{def:S1-action}The \emph{circle} $\mathbb{S}^{1}$ is the nontrivial
real form $(\mathbb{G}_{m,\mathbb{C}},\sigma_{0})$ of $(\mathbb{G}_{m,\mathbb{C}},\sigma_{\mathbb{G}_{m,\mathbb{R}}})$
whose real structure is the composition of the involution $t\mapsto t^{-1}$
of $\mathbb{G}_{m,\mathbb{C}}=\mathrm{Spec}(\mathbb{C}[t^{\pm1}])$
with the complex conjugation. It is a real algebraic group isomorphic
to the group 

\[
SO_{2}\left(\mathbb{R}\right)=\mathrm{Spec}(\mathbb{C}[t^{\pm1}]^{\sigma_{0}^{*}})\cong\mathrm{Spec}(\mathbb{R}[u,v]/(u^{2}+v^{2}-1)),
\]
with group law given by $\left(u,v\right)\cdot(u',v')=(uu'-vv',uv'+u'v)$. 

An action of $\mathbb{S}^{1}$ on a real algebraic variety $(V,\sigma)$
is a real action of $(\mathbb{G}_{m,\mathbb{C}},\sigma_{0})$ on $(V,\sigma)$,
that is, an action $\mu:\mathbb{G}_{m,\mathbb{C}}\times V\rightarrow V$
of $\mathbb{G}_{m,\mathbb{C}}$ on $V$ for which the following diagram
commutes \[\xymatrix{\mathbb{G}_{m,\mathbb{C}}\times V \ar[d]_{\sigma_0\times \sigma} \ar[r]^-{\mu} & V \ar[d]^{\sigma} \\ \mathbb{G}_{m,\mathbb{C}}\times V \ar[r]^-{\mu} & V.}\] 

Let $\pi:(V,\sigma)\rightarrow(C,\tau)$ be an affine morphism between
real algebraic varieties and let $\mu:\mathbb{G}_{m,\mathbb{C}}\times V\rightarrow V$
be an $\mathbb{S}^{1}$-action on $(V,\sigma)$ by real $(C,\tau)$-automorphisms.
Putting $\mathcal{A}=\pi_{*}\mathcal{O}_{V}$, $\mu$ is uniquely
determined by its associated $\mathcal{O}_{C}$-algebra co-action
homomorphism 
\[
\mu^{*}:\mathcal{A}\rightarrow\mathcal{A}\otimes_{\mathcal{O}_{C}}\mathcal{O}_{C}[t^{\pm1}].
\]
The latter determines a $\mathbb{Z}$-grading of $\mathcal{A}$ by
its $\mathcal{O}_{C}$-submodules 
\[
\mathcal{A}_{m}=\left\{ f\in\mathcal{A},\,\mu^{*}f=f\otimes t^{m}\right\} ,\;m\in\mathbb{Z},
\]
of semi-invariants germs of sections of weight $m$. 

The action $\mu$ is said to be \emph{effective} if the set $\{m\in\mathbb{Z},\,\mathcal{A}_{m}\neq\{0\}\}$
is not contained in a proper subgroup of $\mathbb{Z}$, and \emph{hyperbolic}
if there exists $m<0$ and $m'>0$ such that $\mathcal{A}_{m}$ and
$\mathcal{A}_{m'}$ are non zero. The following lemma is an extension
in the relative affine setting of \cite[Lemma 1.7]{DuLi18}. 
\end{defn}

\begin{lem}
\label{lem:Hyperbolic-decomposition} Let $\pi:(V,\sigma)\rightarrow(C,\tau)$
be an affine morphism between real algebraic varieties and let $\mu\colon\mathbb{G}_{m,\mathbb{C}}\times V\rightarrow V$
be an effective $\mathbb{S}^{1}$-action on $(V,\sigma)$ by $(C,\tau)$-automorphisms.
Let $\mathcal{A}=\bigoplus_{m\in\mathbb{Z}}\mathcal{A}_{m}$ be the
corresponding decomposition of the quasi-coherent $\mathcal{O}_{C}$-algebra
$\mathcal{A}=\pi_{*}\mathcal{O}_{V}$ into semi-invariants $\mathcal{O}_{C}$-submodules.
Then the following hold:

1) The action $\mu$ is hyperbolic and $\sigma^{*}\mathcal{A}_{m}=\tau_{*}\mathcal{A}_{-m}$
for every $m\in\mathbb{Z}$.

2) The $\mathcal{O}_{C}$-module $\mathcal{A}_{0}$ is a quasi-coherent
$\mathcal{O}_{C}$-subalgebra of finite type of $\mathcal{A}$. Furthermore,
the restriction $\sigma^{*}:\mathcal{A}_{0}\rightarrow\tau_{*}\mathcal{A}_{0}$
of $\sigma^{*}$ is the comorphism of a real structure $\tau_{0}$
on $\mathrm{Spec}_{C}(\mathcal{A}_{0})$ for which the induced morphism
$\pi_{0}:(\mathrm{Spec}_{C}(\mathcal{A}_{0}),\tau_{0})\rightarrow(C,\tau)$
is a real morphism. 
\end{lem}

\begin{proof}
Since $\pi:(V,\sigma)\rightarrow(C,\tau)$ is an affine morphism,
$\sigma$ is equivalenty determined by its comorphism $\sigma^{*}:\mathcal{A}\rightarrow\tau_{*}\mathcal{A}$.
Since $\mu$ is a non trivial action, there exists a nonzero element
$m\in\mathbb{Z}$ such that $\mathcal{A}_{m}\neq\{0\}$. The commutativity
of the diagram in Definition \ref{def:S1-action} implies that for
a local section $f$ of $\mathcal{A}_{m}$, 
\[
\mu^{*}(\sigma^{*}(f))=(\sigma^{*}\otimes\sigma_{0}^{*})(f\otimes t^{m})=\sigma^{*}(f)\otimes t^{-m}
\]
hence that $\sigma^{*}(f)\in\tau_{*}\mathcal{A}_{-m}$. Thus $\sigma^{*}\mathcal{A}_{m}\subseteq\tau_{*}\mathcal{A}_{-m}$,
and since $(\tau_{*}\sigma^{*})\circ\sigma^{*}=\mathrm{id}_{\mathcal{A}}$,
it follows that the equality $\sigma^{*}\mathcal{A}_{m}=\tau_{*}\mathcal{A}_{-m}$
holds. This shows that the action $\mu$ is hyperbolic and that $\sigma^{*}\mathcal{A}_{0}=\tau_{*}\mathcal{A}_{0}$.
The fact that $\mathcal{A}_{0}$ is a quasi-coherent $\mathcal{O}_{C}$-algebra
of finite type is well-known \cite[Theorem 1.1]{GIT}, and the fact
that $\sigma^{*}|_{\mathcal{A}_{0}}$ is the comomorphism of a real
structure $\tau_{0}$ on $\mathrm{Spec}_{C}(\mathcal{A}_{0})$ making
$\pi_{0}:(\mathrm{Spec}_{C}(\mathcal{A}_{0}),\tau_{0})\rightarrow(C,\tau)$
into a real morphism is a straightforward consequence of the definitions. 
\end{proof}
\begin{defn}
In the setting of Lemma \ref{lem:Hyperbolic-decomposition}, the real
affine morphism $(V,\sigma)\rightarrow(\mathrm{Spec}_{C}(\mathcal{A}_{0}),\tau_{0})$
is called the real (categorical) quotient morphism of the $\mathbb{S}^{1}$-action
on $(V,\sigma)$. 
\end{defn}

\subsection{Principal homogeneous $\mathbb{S}^{1}$-bundles }

Recal that a $\mathbb{G}_{m,\mathbb{C}}$\emph{-torsor} over a complex
algebraic variety $C$ is a $C$-scheme $\rho:P\rightarrow C$ endowed
with an action $\mu:\mathbb{G}_{m,\mathbb{C}}\times P\rightarrow P$
of $\mathbb{G}_{m,\mathbb{C}}$ by $C$-scheme automorphisms, such
that $P$ is Zariski locally isomorphic over $C$ to $C\times\mathbb{G}_{m,\mathbb{C}}$
on which $\mathbb{G}_{m,\mathbb{C}}$ acts by translations on the
second factor. 
\begin{defn}
\label{def:S1-torsor}An $\mathbb{S}^{1}$\emph{-torsor }(also called
a principal homogeneous $\mathbb{S}^{1}$-bundle) over a real algebraic
variety $(C,\tau)$ is a real algebraic variety $\rho:(P,\sigma)\rightarrow(C,\tau)$
endowed with an $\mathbb{S}^{1}$-action $\mu:\mathbb{G}_{m,\mathbb{C}}\times P\rightarrow P$
for which $\rho:P\rightarrow V$ is a $\mathbb{G}_{m,\mathbb{C}}$-torsor. 
\end{defn}

Recall that isomorphism classes of $\mathbb{G}_{m,\mathbb{C}}$-torsors
$\rho:P\rightarrow C$ over $C$ are in one-to-one correspondence
with elements of the Picard group $\mathrm{Pic}(C)\cong H^{1}(C,\mathcal{O}_{C}^{*})$
of $C$. More explicitly, for every such $P$, there exists an invertible
$\mathcal{O}_{C}$-submodule $\mathcal{L}$ of the sheaf of rational
functions $\mathcal{K}_{C}$ of $C$ and an isomorphism of $\mathbb{Z}$-graded
algebras 
\[
\rho_{*}\mathcal{O}_{P}\cong\bigoplus_{m\in\mathbb{Z}}\mathcal{L}^{\otimes m},
\]
where for $m<0$, $\mathcal{L}^{\otimes m}$ denotes the $-m$-th
tensor power of the dual $\mathcal{L}^{\vee}$ of $\mathcal{L}$.
Furthermore, two invertible $\mathcal{O}_{C}$-submodules of $\mathcal{K}_{C}$
determine isomorphic $\mathbb{G}_{m,\mathbb{C}}$-torsors if and only
if they are isomorphic. For $\mathbb{S}^{1}$-torsors, we have the
following counterpart: 
\begin{lem}
\label{lem:S1-torsors} For every $\mathbb{S}^{1}$-torsor $\rho:(P,\sigma)\rightarrow(C,\tau)$
there exists a pair $(\mathcal{L},h)$ consisting of an invertible
$\mathcal{O}_{C}$-submodule $\mathcal{L}\subset\mathcal{K}_{C}$
and a nonzero real rational function $h$ on $(C,\tau)$ such that
$\rho_{*}\mathcal{O}_{P}=\bigoplus_{m\in\mathbb{Z}}\mathcal{L}^{\otimes m}$
and $\mathcal{L}\otimes\tau^{*}\mathcal{L}=h^{-1}\mathcal{O}_{C}$
as $\mathcal{O}_{C}$-submodules of $\mathcal{K}_{C}$. 

Furthermore, two such pairs $(\mathcal{L}_{1},h_{1})$ and $(\mathcal{L}_{2},h_{2})$
determine isomorphic $\mathbb{S}^{1}$-torsors if and only if there
exists a rational function $f\in\Gamma(C,\mathcal{K}_{C}^{*})$ such
that $\mathcal{L}_{1}^{\vee}\otimes\mathcal{L}_{2}=f^{-1}\mathcal{O}_{C}$
and $h_{2}=(f\tau^{*}f)h_{1}$. 
\end{lem}

\begin{proof}
Let $\mathcal{A}=\bigoplus_{m\in\mathbb{Z}}\mathcal{A}_{m}$ be the
decomposition of $\mathcal{A}=\rho_{*}\mathcal{O}_{P}$ into $\mathcal{O}_{C}$-submodules
of semi-invariants with respect to the action $\mu$ and let $\mathcal{L}$
be an invertible $\mathcal{O}_{C}$-submodule of $\mathcal{K}_{C}$
for which we have an isomorphism of graded $\mathcal{O}_{C}$-algebras
\[
\Psi:\mathcal{A}=\bigoplus_{m\in\mathbb{Z}}\mathcal{A}_{m}\stackrel{\cong}{\rightarrow}\bigoplus_{m\in\mathbb{Z}}\mathcal{L}^{\otimes m}.
\]
By Lemma \ref{lem:Hyperbolic-decomposition} 1), we have $\sigma^{*}\mathcal{A}_{m}=\tau_{*}\mathcal{A}_{-m}$
for every $m\in\mathbb{Z}$. It follows that for every $m\in\mathbb{Z}$,
the composition 
\[
\varphi_{m}:\tau_{*}\Psi\circ\sigma^{*}\circ\Psi^{-1}:\mathcal{L}^{\otimes m}\rightarrow\tau_{*}\mathcal{L}^{\otimes-m}
\]
is an isomorphism of $\mathcal{O}_{C}$-modules such that $\varphi_{0}=\tau^{*}:\mathcal{O}_{C}=\mathcal{L}^{\otimes0}\rightarrow\tau_{*}\mathcal{L}^{\otimes-0}=\tau_{*}\mathcal{O}_{C}$
and 
\[
\varphi_{m+m'}=\varphi_{m}\otimes\varphi_{m'}:\mathcal{L}^{\otimes(m+m')}=\mathcal{L}^{\otimes m}\otimes\mathcal{L}^{\otimes m'}\rightarrow\tau_{*}\mathcal{L}^{\otimes(-m+-m')}=\tau_{*}\mathcal{L}^{\otimes(-m)}\otimes\tau_{*}\mathcal{L}^{\otimes(-m')}
\]
for every $m,m'\in\mathbb{Z}$. Furthermore, since $\tau_{*}\sigma^{*}\circ\sigma^{*}=\mathrm{id}_{\mathcal{A}}$
and $\tau_{*}^{2}=\mathrm{id}_{\mathcal{O}_{C}}$, we have 
\[
\tau_{*}\varphi_{-m}\circ\varphi_{m}=\mathrm{id}_{\mathcal{L}^{\otimes m}}\quad\textrm{and}\quad\tau_{*}\varphi_{m}\circ\varphi_{-m}=\mathrm{id}_{\mathcal{L}^{\otimes(-m)}}
\]
for every $m\in\mathbb{Z}$. For $m=1$ and $m'=-1$, the commutativity
of the diagram \[\xymatrix{\mathcal{L}\otimes \mathcal{L}^{\vee} \ar[d]_{\mathrm{ev}} \ar[rr]^{\varphi_1\otimes \varphi_{-1} } && \tau_*\mathcal{L}^{\vee}\otimes \tau_*\mathcal{L} \ar[d]^{\tau_*{\mathrm{ev}}} \\ \mathcal{O}_C \ar[rr]^{\tau^*} && \tau_*\mathcal{O}_C,}\]
where $\mathrm{ev}:\mathcal{L}\otimes\mathcal{L}^{\vee}\stackrel{\cong}{\rightarrow}\mathcal{O}_{C}$
is the canonical homomorphism $f\otimes f'\mapsto f'(f)$, implies
that 
\[
\varphi_{-1}=(^{t}\varphi_{1})^{-1}:\mathcal{L}^{\vee}\rightarrow\tau_{*}\mathcal{L}.
\]
It follows that $\varphi_{m}=\varphi_{1}^{\otimes m}$ for $m\geq1$
and that $\varphi_{m}=(^{t}\varphi_{1})^{\otimes(-m)}$ for $m\leq-1$.
The collection $(\varphi_{m})_{m\in\mathbb{Z}}$ is thus uniquely
determined by $\varphi_{0}=\tau^{*}$ and an isomorphism $\varphi_{1}:\mathcal{L}\rightarrow\tau_{*}\mathcal{L}^{\vee}$
satisfying the identity $\tau_{*}(^{t}\varphi_{1})^{-1}\circ\varphi_{1}=\mathrm{id}_{\mathcal{L}}$,
hence equivalently by an isomorphism $\psi=\tau^{*}\varphi_{1}:\tau^{*}\mathcal{L}\rightarrow\mathcal{L}^{\vee}$
such that $(^{t}\psi^{-1})\circ\tau^{*}\psi=\mathrm{id}_{\mathcal{L}}$.
An isomorphism $\psi:\tau^{*}\mathcal{L}\stackrel{\cong}{\rightarrow}\mathcal{L}^{\vee}$
is in turn equivalenty determined by an isomorphism $\mathcal{O}_{C}\stackrel{\cong}{\rightarrow}\mathcal{L}\otimes\tau^{*}\mathcal{L}$,
that is, by a rational function $h\in\Gamma(C,\mathcal{K}_{C}^{*})$
such that $\mathcal{L}\otimes\tau^{*}\mathcal{L}=h^{-1}\mathcal{O}_{C}$
as $\mathcal{O}_{C}$-submodules of $\mathcal{K}_{C}$. The condition
$(^{t}\psi^{-1})\circ\tau^{*}\psi=\mathrm{id}_{\mathcal{L}}$ then
amounts to the property that $h^{-1}(\tau^{*}h)=1$, i.e. that $h$
is a real rational function on $(C,\tau)$.

Two invertible $\mathcal{O}_{C}$-submodules $\mathcal{L}_{1},\mathcal{L}_{2}\subset\mathcal{K}_{C}$
define equivariantly isomorphic $\mathbb{G}_{m,\mathbb{C}}$-torsors
$\rho_{1}:P_{1}\rightarrow C$ and $\rho_{2}:P_{2}\rightarrow C$
if and only if there exists an isomorphism $\alpha:\mathcal{L}_{1}\stackrel{\cong}{\rightarrow}\mathcal{L}_{2}$.
When $\mathcal{L}_{1}$ and $\mathcal{L}_{2}$ come with respective
isomorphisms $\psi_{1}:\tau^{*}\mathcal{L}_{1}\rightarrow\mathcal{L}_{1}^{\vee}$
and $\psi_{2}:\tau^{*}\mathcal{L}_{2}\rightarrow\mathcal{L}_{2}^{\vee}$
corresponding to $\mathbb{S}^{1}$-actions on $(P_{1},\sigma_{1})$
and $(P_{2},\sigma_{2})$, the condition that a given isomorphism
$\alpha:\mathcal{L}_{1}\stackrel{\cong}{\rightarrow}\mathcal{L}_{2}$
induces an $\mathbb{S}^{1}$-equivariant isomorphism between $(P_{1},\sigma_{1})$
and $(P_{2},\sigma_{2})$ is equivalent to the commutativity of the
diagram \[\xymatrix{ \tau^*\mathcal{L}_1 \ar[d]_{\tau^*\alpha} \ar[r]^{\psi_1} & \mathcal{L}_1^{\vee} \ar[d]^{^t\alpha^{-1}} \\ \tau^*\mathcal{L}_2 \ar[r]^{\psi_2} & \mathcal{L}_2^{\vee}.}\]
The isomorphism $\alpha$ is uniquely determined by a rational function
$f\in\Gamma(C,\mathcal{K}_{C}^{*})$ such that $\mathcal{L}_{1}^{\vee}\otimes\mathcal{L}_{2}=f^{-1}\mathcal{O}_{C}$
as $\mathcal{O}_{C}$-submodules of $\mathcal{K}_{C}$. By definition
of $h_{1}$ and $h_{2}$ as the unique nonzero real rational functions
on $(C,\tau)$ such that $\mathcal{L}_{i}\otimes\tau^{*}\mathcal{L}_{i}=h_{i}^{-1}\mathcal{O}_{C}$,
$i=1,2$, the commutativity of the above diagram is equivalent to
the equality $h_{2}=(f\tau^{*}f)h_{1}$. 
\end{proof}
\begin{example}
\label{exa:S1-torsor-pt} By Hilbert\textquoteright s Theorem 90,
every $\mathbb{G}_{m,\mathbb{C}}$-torsor over $\mathrm{Spec}(\mathbb{C})$
is isomorphic to the trivial one, that is, to $\mathbb{G}_{m,\mathbb{C}}$
acting on itself by translations. In contrast, there exists precisely
two non-isomorphic $\mathbb{S}^{1}$-torsors over $\mathrm{Spec}(\mathbb{R})=(\mathrm{Spec}(\mathbb{C}),\sigma_{\mathrm{Spec}(\mathbb{R})})$: 

1) The trivial one given by $\mathbb{S}^{1}=(\mathbb{G}_{m,\mathbb{C}}=\mathrm{Spec}(\mathbb{C}[t^{\pm1}],\sigma_{0})$
acting on itself by translations. A corresponding pair is $(\mathcal{L},h)=(\mathbb{C},1)$, 

2) A nontrivial one $\hat{\mathbb{S}}^{1}=(\mathrm{Spec}(\mathbb{C}[u^{\pm1}],\hat{\sigma}_{0})$
whose real structure $\hat{\sigma}_{0}$ is the composition of the
involution $u\mapsto-u^{-1}$ with the complex conjugation, endowed
with the $\mathbb{S}^{1}$-action given by $t\cdot u=tu$. A corresponding
pair is $(\mathcal{L},h)=(\mathbb{C},-1)$. 

Note that the real locus of $\mathbb{S}^{1}$ is isomorphic to the
real circle $S^{1}=\{x^{2}+y^{2}=1\}\subset\mathbb{R}^{2}$ whereas
the real locus of $\hat{\mathbb{S}}^{1}$ is empty. 
\end{example}

\section{Circle actions on smooth real affine surfaces}

In this section, we first review the correspondence between normal
real affine surfaces $(S,\sigma)$ with effective $\mathbb{S}^{1}$-actions
and suitable pairs consisting of a Weil $\mathbb{Q}$-divisor and
a rational function on smooth real affine curves $(C,\tau)$, which
we call real DPD-pairs. We characterize smooth affine surfaces in
terms of properties of their corresponding pairs. We then describe
the structure of exceptional orbits of $\mathbb{S}^{1}$-actions on
smooth surfaces $(S,\sigma)$ in relation to degenerate fibers of
their quotient morphisms. 

\subsection{Real DPD-presentations of smooth real affine surfaces with $\mathbb{S}^{1}$-actions }

Recall that a Weil $\mathbb{Q}$-divisor on a smooth real affine curve
$(C=\mathrm{Spec}(A_{0}),\tau)$ is an element of the abelian group
consisting of formal sums
\[
D=\sum_{c\in C}D(c)\{c\}\in\mathbb{Q}\otimes_{\mathbb{Z}}\mathrm{Div}(C)
\]
such that $D(c)\in\mathbb{Q}$ is equal to zero for all but finitely
many points $c\in C$. The \emph{support} of $D$ is the finite set
of points $c\in C$ such that $D(c)\neq0$. The group of Weil $\mathbb{Q}$-divisors
is partially ordered by the relation 
\[
(D\leq D'\Leftrightarrow D(c)\leq D'(c)\quad\forall c\in C).
\]
Every nonzero rational function $f$ on $C$ determines a Weil $\mathbb{Q}$-divisor
$\mathrm{div}(f)=\sum_{c\in C}(\mathrm{ord}_{c}f)\{c\}$ with integral
coefficients. For every Weil $\mathbb{Q}$-divisor $D$ on $C$, we
denote by $\Gamma(C,\mathcal{O}_{C}(D))$ the $A_{0}$-submodule of
the field of fractions $\mathrm{Frac}(A_{0})$ of $A_{0}$ generated
by nonzero rational functions $f$ on $C$ such that $\mathrm{div}(f)+D\geq0$. 

Given an automorphism $\alpha$ of $C$ as a scheme over $\mathbb{R}$
or $\mathbb{C}$, the pull-back of $D=\sum_{c\in C}D(c)\{c\}$ by
$\alpha$ is the Weil $\mathbb{Q}$-divisor 
\[
\alpha^{*}D=\sum_{c\in C}D(c)\{\alpha^{-1}(c)\}=\sum_{c\in C}D(\alpha(c))\{c\}.
\]

\begin{defn}
A \emph{real DPD-pair} on a smooth real affine curve $(C,\tau)$ is
a pair $(D,h)$ consisting of a Weil $\mathbb{Q}$-divisor $D$ on
$C$ and a nonzero real rational function $h$ on $(C,\tau)$ satisfying
$D+\tau^{*}D\leq\mathrm{div}(h)$. 
\end{defn}

We say that two rational numbers $r_{i}=p_{i}/q_{i}$, $i=1,2$, where
$\gcd(p_{i},q_{i})=1$, form a \emph{regular pair} if $\left|p_{1}q_{2}-p_{2}q_{1}\right|=1$. 
\begin{defn}
A real DPD-pair $(D,h)$ on a smooth real affine curve $(C,\tau)$
is said to be \emph{regular} if for every $c\in C$ such that $D(c)+D(\tau(c))<\mathrm{ord}_{c}(h)$
the rational numbers $D(c)$ and $D(\tau(c))-\mathrm{ord}_{c}(h)$
form a regular pair. 
\end{defn}

Given a smooth real affine curve $(C,\tau)$, a pair $(\mathcal{L},h)$
consisting of an invertible $\mathcal{O}_{C}$-submodule $\mathcal{L}\subset\mathcal{K}_{C}$
and a real rational function $h$ on $(C,\tau)$ such that $\mathcal{L}\otimes\tau^{*}\mathcal{L}=h^{-1}\mathcal{O}_{C}$
as $\mathcal{O}_{C}$-submodules of $\mathcal{K}_{C}$ determines
a Cartier divisor $D$ on $C$ such that $D+\tau^{*}D=\mathrm{div}(h)$,
hence a regular real DPD-pair $(D,h)$ on $(C,\tau)$. By Lemma \ref{lem:S1-torsors},
every smooth real affine surface $(S,\sigma)$ endowed with the structure
of an $\mathbb{S}^{1}$-torsor over $(C,\tau)$ is determined by such
a regular real DPD-pair $(D,h)$. More generally, we have the following: 
\begin{thm}
\label{thm:MainThm} Every normal real affine surface $(S,\sigma)$
with an effective $\mathbb{S}^{1}$-action $\mu:\mathbb{G}_{m,\mathbb{C}}\times S\rightarrow S$
is determined by a smooth real affine curve $(C,\tau)$ and a real
DPD-pair $(D,h)$ on it. Furthermore, the following hold:

1) Two DPD-pairs $(D_{1},h_{1})$ and $(D_{2},h_{2})$ on the same
curve $(C,\tau)$ determine $\mathbb{S}^{1}$-equivariantly isomorphic
real affine surfaces if and only if there exists a real automorphism
$\psi$ of $(C,\tau)$ and a rational function $f$ on $C$ such that
\[
\psi^{*}D_{2}=D_{1}+\mathrm{div}(f)\qquad\textrm{and}\qquad\psi^{*}h_{2}=(f\tau^{*}f)h_{1}.
\]

2) The normal real affine surface $(S,\sigma)$ determined by a real
DPD-pair $(D,h)$ is smooth if and only if the pair is regular.
\end{thm}

The correspondence between normal real affine surface $(S,\sigma)$
with an effective $\mathbb{S}^{1}$-actions and real DPD-pairs on
smooth real affine curves $(C,\tau)$ was established in \cite[Proposition 3.2]{DuLi18}
as a particular case of a general correspondence between $\mathbb{S}^{1}$-actions
on normal real affine varieties and suitable pairs $(D,h)$ on certain
normal real semi-projective varieties \cite[Corollary 2.16]{DuLi18},
whose proof uses the formalism of polyhedral divisors due to Altmann
and Hausen \cite{AH}. Since this correspondence provides an explicit
method to determine the data $(S,\sigma)$ and $(C,\tau)$, $(D,h)$
from each others, we will review it in detail using the DPD-formalism
of \cite{FZ03} in the next subsections. 
\begin{proof}[Proof of Theorem \ref{thm:MainThm}]
 Assertion 1) follows from Corollary 2.16 in \cite{DuLi18}. Note
that if $(D_{2},h_{2})$ is regular real DPD-pair then for every real
automorphism $\psi$ of $(C,\tau)$ and every rational function $f$
on $C$, the real DPD-pair 
\[
(D_{1},h_{1})=(\psi^{*}D_{2}-\mathrm{div}(f),(f\tau^{*}f)^{-1}\psi^{*}h_{2})
\]
is regular due to the fact that $\mathrm{div}(f)$ and $\mathrm{div}((f\tau^{*}f)^{-1}\psi^{*}h_{2})$
are integral Weil divisors on $C$. 

To prove 2), let $(S,\sigma)$ be the normal real affine surface with
$\mathbb{S}^{1}$-action determined by real DPD-pair $(D,h)$ on a
smooth real affine curve $(C,\tau)$ as in $\S$ \ref{subsec:DPD-2-Surface}
below. Let $\pi:(S,\sigma)\rightarrow(C,\tau)$ be its real quotient
morphism and let $D_{+}=D$ and $D_{-}=\tau^{*}D-\mathrm{div}(h)$.
By \cite[Theorem 4.15]{FZ03}, the singular locus of $S$ is contained
in the fibers of the quotient morphism $\pi:S\rightarrow C$ over
the points $c\in C$ such that $D_{+}(c)+D_{-}(c)<0$. Furthermore,
for such a point $c$, $S$ is smooth at every point of $\pi^{-1}(c)$
if and only if the rational numbers $D_{+}(c)$ and $D_{-}(c)$ form
a regular pair. 
\end{proof}

\subsubsection{\label{subsec:DPD-2-Surface}From real DPD-pairs to normal real affine
surfaces with effective $\mathbb{S}^{1}$-actions }

Given a real DPD-pair $(D,h)$ on a smooth real affine curve $(C=\mathrm{Spec}(A_{0}),\tau)$,
we set $D_{+}=D$ and $D_{-}=\tau^{*}D_{+}-\mathrm{div}(h)$. The
condition $D+\tau^{*}D\leq\mathrm{div}(h)$ implies that $D_{+}+D_{-}\leq0$,
so that for every $m'\leq0\leq m$, the product 
\[
\Gamma(C,\mathcal{O}_{C}(-m'D_{-}))\cdot\Gamma(C,\mathcal{O}_{C}(mD_{+}))
\]
in $\mathrm{Frac}(A_{0})$ in contained either in $\Gamma(C,\mathcal{O}_{C}(-(m'+m)D_{-}))$
if $m'+m\leq0$ or in $\Gamma(C,\mathcal{O}_{C}((m'+m)D_{+}))$ $\textrm{if }m'+m\geq0$.
It follows that the graded $A_{0}$-module 
\[
A_{0}[D_{-},D_{+}]=\bigoplus_{m<0}\Gamma(C,\mathcal{O}_{C}(-mD_{-}))\oplus\Gamma(C,\mathcal{O}_{C})\oplus\bigoplus_{m>0}\Gamma(C,\mathcal{O}_{C}(mD_{+}))
\]
is a graded $A_{0}$-algebra for the multiplication law given by component
wise multiplication in $\mathrm{Frac}(A_{0})$. By \cite[$\S$ 4.2]{FZ03},
$A_{0}[D_{+},D_{-},h]$ is finitely generated over $\mathbb{C}$ and
normal. The grading then corresponds to an effective hyperbolic $\mathbb{G}_{m,\mathbb{C}}$-action
$\mu:\mathbb{G}_{m,\mathbb{C}}\times S\rightarrow S$ on the normal
complex affine surface $S=\mathrm{Spec}(A_{0}[D_{-},D_{+}])$. The
ring of invariants for this action is equal to $A_{0}$ and the morphism
$\pi:S\rightarrow C=\mathrm{Spec}(A_{0})$ induced by the inclusion
$A_{0}\subset A_{0}[D_{-},D_{+}]$ is the categorical quotient morphism.
Since $D_{-}=\tau^{*}D_{+}-\mathrm{div}(h)$, for every $m\geq1$,
the homomorphism 
\[
\tau_{m}^{*}:\Gamma(C,\mathcal{O}_{C}(mD_{+}))\mapsto\Gamma(C,\mathcal{O}_{C}(mD_{-})),\;f\mapsto h^{m}\tau^{*}f
\]
is an isomorphism with inverse
\[
\tau_{-m}^{*}:\Gamma(C,\mathcal{O}_{C}(-mD_{-}))\mapsto\Gamma(C,\mathcal{O}_{C}(-mD_{+})),\;f\mapsto h^{m}\tau^{*}f.
\]
Letting $\tau_{0}=\tau$, these isomorphisms collect into an automorphism
$\sigma^{*}=\bigoplus_{m\in\mathbb{Z}}\tau_{m}^{*}$ of $A_{0}[D_{-},D_{+}]$
which is the comorphism of a real structure $\sigma$ on $S$ for
which we have $\sigma\circ\mu=\mu\circ(\sigma_{0}\times\sigma)$.
It follows that $(S,\sigma)$ is a normal real affine surface and
that $\mu:\mathbb{G}_{m,\mathbb{C}}\times S\rightarrow S$ is an effective
$\mathbb{S}^{1}$-action on $(S,\sigma)$ in the sense of Definition
\ref{def:S1-action}. 
\begin{example}
\label{exa:Reducible-Real-Fiber}Let $(C=\mathrm{Spec}(A_{0}),\tau)$
be a smooth real affine curve with a real point $c$ whose defining
ideal is principal, generated by a real regular function $h$ on $(C,\tau)$.
Let $D$ be the trivial divisor $0$. Then $(D,h)$ is a real DPD-pair
on $(C,\tau)$ for which we have $D_{+}=D=0$ and $D_{-}=\tau^{*}D_{-}-\mathrm{div}(h)=-\{c\}$.
It follows that $\Gamma(C,\mathcal{O}_{C}(mD_{+}))=A_{0}$ and that
$\Gamma(C,\mathcal{O}_{C}(mD_{-}))=h^{m}A_{0}$ for every $m\geq0$.
The corresponding homomorphism 
\[
\tau_{m}^{*}:\Gamma(C,\mathcal{O}_{C}(mD_{+}))=A_{0}\mapsto h^{m}A_{0}=\Gamma(C,\mathcal{O}_{C}(mD_{-}))
\]
is the multiplication by $h^{m}$. The algebra $A_{0}[D_{+},D_{-}]$
is generated by the homogeneous elements 
\[
x=1\in\Gamma(C,\mathcal{O}_{C}(D_{+}))=\Gamma(C,\mathcal{O}_{C})\quad\textrm{and}\quad y=h\in\Gamma(C,\mathcal{O}_{C}(D_{-}))=\Gamma(C,\mathcal{O}_{C}(-\{c\}))
\]
of degree $1$ and $-1$ respectively. These satisfy the obvious homogeneous
relation $xy=h$, and we have 
\[
A_{0}[D_{+},D_{-}]\cong A_{0}[x,y]/(xy-h).
\]
The corresponding $\mathbb{G}_{m,\mathbb{C}}$-action $\mu$ on $S=\mathrm{Spec}(A_{0}[D_{+},D_{-}])$
is given by $t\cdot(x,y)=(tx,t^{-1}y)$ and the real structure $\sigma$
for which $\mu$ becomes an $\mathbb{S}^{1}$-action on $(S,\sigma)$
is the lift of $\tau$ defined by $\sigma^{*}x=y$ and $\sigma^{*}y=x$. 
\end{example}

\subsubsection{\label{subsec:Surface-2-DPD}From normal real affine surfaces with
effective $\mathbb{S}^{1}$-actions to real DPD-pairs}

Given a normal real affine surface $(S,\sigma)$ with an effective
$\mathbb{S}^{1}$-action $\mu:\mathbb{G}_{m,\mathbb{C}}\times S\rightarrow S$,
it follows from Lemma \ref{lem:Hyperbolic-decomposition} that the
coordinate ring $A$ of $S$ decomposes as the direct sum $A=\bigoplus_{m\in\mathbb{Z}}A_{m}$
of semi-invariants sub-spaces such that $\sigma^{*}(A_{m})=A_{-m}$
for every $m\in\mathbb{Z}$. The curve $C=\mathrm{Spec}(A_{0})$ is
the categorical quotient of the $\mathbb{G}_{m,\mathbb{C}}$-action
on $S$. The restriction of $\sigma^{*}$ to $A_{0}$ induces a real
structure $\tau$ on $C$. Let $s\in\mathrm{Frac}(A)$ be any semi-invariant
rational function of weight $1$ and let $h=s\sigma^{*}s\in\mathrm{Frac}(A)$.
Since $\sigma^{*}s$ is a semi-invariant rational function of weight
$-1$, $h$ is a $\sigma^{*}$-invariant rational function of weight
$0$, hence a $\tau^{*}$-invariant element of $\mathrm{Frac}(A_{0})$.
For every $m\in\mathbb{Z}$, $s^{-m}A_{m}$ is a locally free $A_{0}$-submodule
of $\mathrm{Frac}(A_{0})$. By \cite[$\S$ 4.2]{FZ03}, there exists
Weil $\mathbb{Q}$-divisors $D_{+}$ and $D_{-}$ on $C$ satisfying
$D_{+}+D_{-}\leq0$ such that for every $m\geq0$ we have 
\[
s^{-m}A_{m}=\Gamma(C,\mathcal{O}_{C}(mD_{+}))\quad\textrm{and}\quad s^{m}A_{-m}=\Gamma(C,\mathcal{O}_{C}(mD_{-}))
\]
as $A_{0}$-submodules of $\mathrm{Frac}(A_{0})$. Since by Lemma
\ref{lem:Hyperbolic-decomposition} and the definition of $h$, we
have 
\[
\tau^{*}(s^{-m}\cdot A_{m})=h^{-m}(s^{m}\cdot A_{-m})\quad\forall m\in\mathbb{Z},
\]
it follows that $D_{-}=\tau^{*}D_{+}-\mathrm{div}(h)$. So setting
$D=D_{+}$, the pair $(D,h)$ is a real DPD- pair on the smooth real
affine curve $(C,\tau)$. By construction, $S\cong\mathrm{Spec}(A_{0}[D_{-},D_{+}])$
and the real structure $\sigma$ on $S$ coincides with that constructed
from $(D,h)$ in the previous subsection. 
\begin{example}
Let $(C=\mathrm{Spec}(R),\tau)$ be a smooth real affine curve with
a real point $c$ whose defining ideal is principal, generated by
some real regular function $f$ on $(C,\tau)$. Let $A=R[x^{\pm1},y]/(xy^{2}-f)$
and let $S=\mathrm{Spec}(A)$. The morphism $\mu:\mathbb{G}_{m,\mathbb{C}}\times S\rightarrow S$,
$(t,(x,y))\mapsto(t^{2}x,t^{-1}y)$ defines an $\mathbb{G}_{m,\mathbb{C}}$-action
on $S$ by $C$-automorphisms, which becomes an $\mathbb{S}^{1}$-action
by $(C,\tau)$-automorphisms when $S$ is endowed with the unique
real structure $\sigma$ lifting $\tau$ such $\sigma^{*}x=x^{-1}$
and $\sigma^{*}y=xy$. The ring of $\mathbb{G}_{m,\mathbb{C}}$-invariant
$A_{0}$ is equal to $R[xy^{2}]/(xy^{2}-f)\cong R$. Choosing $s=y^{-1}$
as semi-invariant rational function of weight $1$, we have $h=y^{-1}\sigma^{*}(y^{-1})=x^{-1}y^{-2}=f^{-1}\in\mathrm{Frac}(R)$.
The decompositon of $A$ into subspaces of semi-invariants functions
is then given for every $m\geq0$ by 
\begin{align*}
s^{-m}A_{m} & =s^{-m}R\cdot(xy)^{m}=R\cdot(xy^{2})^{m}=f^{m}R=\Gamma(C,\mathcal{O}_{C}(mD_{+}),\\
s^{2m+1}A_{-2m-1} & =s^{2m+1}R\cdot(x^{-m}y)=R\cdot(xy^{2})^{-m}=f^{-m}R=\Gamma(C,\mathcal{O}_{C}((2m+1)D_{-})),\\
s^{2m}A_{-2m} & =s^{2m}R\cdot x^{-m}=R\cdot(x^{-m}y^{-2m})=f^{-m}R=\Gamma(C,\mathcal{O}_{C}(2mD_{-})).
\end{align*}
It follows that $D_{+}=-r\{c\}$ for some rational number $r\in]0,1]$
and that
\[
D_{-}=\tau^{*}(D_{+})-\mathrm{div}(h)=r\{c\}-\mathrm{div}(f^{-1})=(1-r)\{c\}.
\]
Since $f^{-m}R=\Gamma(C,\mathcal{O}_{C}((2m+1)D_{-}))=\Gamma(C,\mathcal{O}_{C}(2mD_{-}))$
for every $m\geq0$, it follows that 
\[
\frac{m}{2m+1}\leq(1-r)<\frac{m+1}{2m+1}\quad\textrm{and}\quad\frac{1}{2}\leq(1-r)<\frac{m+1}{2m}
\]
for every $m\geq0$. Thus $(1-r)=\frac{1}{2}$ and a real DPD-pair
on $(C,\tau)$ corresponding to $(S,\sigma)$ endowed with the $\mathbb{S}^{1}$-action
$\mu$ is $(D,h)=(-\frac{1}{2}\{c\},f^{-1})$. 
\end{example}

\subsection{Real fibers of the quotient morphism: principal and exceptional orbits}

Let $(S,\sigma)$ be a smooth real affine surface with an effective
$\mathbb{S}^{1}$-action $\mu:\mathbb{G}_{m,\mathbb{C}}\times S\rightarrow S$,
and let $\pi:(S,\sigma)\rightarrow(C,\tau)$ be its real quotient
morphism. Recall that $\pi:S\rightarrow C=\mathrm{Spec}(\Gamma(S,\mathcal{O}_{S})^{\mathbb{G}_{m,\mathbb{C}}})$
is surjective and that each fiber of $\pi$ contains a unique closed
$\mathbb{G}_{m,\mathbb{C}}$-orbit $Z$ and is the union of all $\mathbb{G}_{m,\mathbb{C}}$-orbits
in $S$ containing $Z$ in their closure. In the complex case, \cite[Theorem 18]{FZ03}
provides a description of the structure of the fibers of $\pi$ in
terms of a pair of Weil $\mathbb{Q}$-divisors $D_{+}$ and $D_{-}$
on $C$ for which $\Gamma(S,\mathcal{O}_{S})=A_{0}[D_{-},D_{+}]$
(see $\S$ \ref{subsec:Surface-2-DPD}). In this subsection, we specialize
this description to fibers of $\pi$ over points in the real locus
of $(C,\tau)$. We begin with the following example which illustrates
different possibilities for such fibers. 
\begin{example}
\label{exa:Real-Fiber-Struct}Let $S_{\varepsilon}\subset\mathbb{A}_{\mathbb{C}}^{3}=\mathrm{Spec}(\mathbb{C}[x,y,z])$
be the smooth complex affine surface with equation 
\[
xy=z^{2}+\varepsilon,\qquad\textrm{where }\varepsilon=\pm1,
\]
endowed with the real structure $\sigma$ defined as the composition
of the involution $(x,y,z)\mapsto(y,x,z)$ with the complex conjugation.
The effective $\mathbb{G}_{m,\mathbb{C}}$-action $\mu$ on $S_{\varepsilon}$
given by $t\cdot(x,y,z)=(tx,t^{-1}y,z)$ defines an $\mathbb{S}^{1}$-action
on $(S_{\varepsilon},\sigma)$ whose real quotient morphism coincides
with the projection 
\[
\pi=\mathrm{pr}_{z}:(S_{\varepsilon},\sigma)\rightarrow(C,\tau)=(\mathrm{Spec}(\mathbb{C}[z],\sigma_{\mathbb{A}_{\mathbb{R}}^{1}}).
\]
A corresponding real DPD-pair is for instance $(D,h)=(0,z^{2}+\varepsilon)$
where $0$ denotes the trivial Weil divisor. 

The morphism $(x,y,z)\mapsto(-x,-y,-z)$ defines a fixed point free
real action of $\mathbb{Z}_{2}$ on $(S_{\varepsilon},\sigma)$ commuting
with the $\mathbb{S}^{1}$-action. The quotient surface $\overline{S}_{\varepsilon}=S_{\varepsilon}/\mathbb{Z}_{2}$
is smooth and $\sigma$ descends to a real structure $\overline{\sigma}$
on it. The morphism $\pi$ descends to a real morphism $\overline{\pi}:(\overline{S}_{\varepsilon},\overline{\sigma})\rightarrow(\overline{C},\overline{\tau})=(\mathrm{Spec}(\mathbb{C}[z^{2}]),\sigma_{\mathbb{A}_{\mathbb{R}}^{1}})$
which coincides with the real quotient morphism of the induced $\mathbb{S}^{1}$-action
on $\overline{S}_{\varepsilon}$. 

1) If $\varepsilon=1$, then since $z^{2}+1=(z-i)(z+i)=f\tau^{*}f$,
it follows from Theorem \ref{thm:MainThm} 1) and Lemma \ref{lem:S1-torsors}
that $\pi:(S_{1},\sigma)\rightarrow(C,\tau)$ restricts to the trivial
$\mathbb{S}^{1}$-torsor over the principal real affine open subset
$(C_{h}=\mathrm{Spec}(\mathbb{C}[z]_{z^{2}+1}),\tau|_{C_{h}})$ of
$C$. In particular, for every real point $c$ of $(C,\tau)$, $(\pi^{-1}(c),\sigma|_{\pi^{-1}(c)})$
is isomorphic to $\mathbb{S}^{1}$ on which $\mathbb{S}^{1}$-acts
by translations.  

Since the real point $0\in(C,\tau)$ is a fixed point of the $\mathbb{Z}_{2}$-action
on $C$, the fiber of $\overline{\pi}:(\overline{S}_{1},\overline{\sigma})\rightarrow(\overline{C},\overline{\tau})$
over the real point $0\in(\overline{C},\overline{\tau})$ has multiplicity
two. When endowed with its reduced structure, it is isomorphic to
the quotient of $\mathrm{Spec}(\mathbb{C}[x,y]/(xy-1))$ by the involution
$(x,y)\mapsto(-x,-y)$, hence to $\mathbb{A}_{*}^{1}=\mathrm{Spec}(\mathbb{C}[w^{\pm1}])$,
where $w=x^{2}$. The real structure is given by the composition of
the involution $w\mapsto w^{-1}$ with the complex conjugation, and
the group $\mathbb{G}_{m,\mathbb{C}}$ acts on it by $t\cdot w=t^{2}w$.
So $(\overline{\pi}^{-1}(0)_{\mathrm{red}},\overline{\sigma}|_{\overline{\pi}^{-1}(0)_{\mathrm{red}}})$
is isomorphic to $\mathbb{S}^{1}$ on which $\mathbb{S}^{1}$ acts
with stabilizer $\mu_{2}$. 

2) If $\varepsilon=-1$, then, in contrast with the previous case,
there is no rational function $f\in\mathbb{C}(z)$ such that $z^{2}-1=f\tau^{*}f$.
Consequently, there is no real open subset of $(C,\tau)$ over which
$\pi:(S_{-1},\sigma)\rightarrow(C,\tau)$ restricts to the trivial
$\mathbb{S}^{1}$-torsor. For a real point $c$ of $(C,\tau)$, $h=z^{2}-1$
takes negative value at $c$ if $c\in]-1,1[$ and positive value if
$c\in]-\infty,-1[\cup]1,+\infty[$. The fiber $(\pi^{-1}(c),\sigma|_{\pi^{-1}(c)})$
is thus isomorphic to the nontrivial $\mathbb{S}^{1}$-torsor $\hat{\mathbb{S}}^{1}$
of Example \ref{exa:S1-torsor-pt} in the first case, and to the trivial
$\mathbb{S}^{1}$-torsor $\mathbb{S}^{1}$ in the second case. 

The fiber of $\pi$ over the point $\pm1$ is isomorphic to $\mathrm{Spec}(\mathbb{C}[x,y]/(xy))$
and thus consists of two affine lines $\overline{O}^{+}=\mathrm{Spec}(\mathbb{C}[x])$
and $\overline{O}^{-}=\mathrm{Spec}(\mathbb{C}[y])$ exchanged by
the real structure $\sigma$, intersecting at the real point $p=(0,0,0)$
of $(S_{-1},\sigma)$. The curves $O^{\pm}=\overline{O}^{\pm}\setminus\{p\}\cong\mathbb{A}_{\mathbb{C}}^{1}\setminus\{0\}$
endowed with the induced $\mathbb{G}_{m,\mathbb{C}}$-actions are
trivial $\mathbb{G}_{m,\mathbb{C}}$-torsors and $p$ is an $\mathbb{S}^{1}$-fixed
point. 

As in the previous case, since the real point $0\in(C,\tau)$ is a
fixed point of the $\mathbb{Z}_{2}$-action on $C$, the fiber of
$\overline{\pi}:(\overline{S}_{-1},\overline{\sigma})\rightarrow(\overline{C},\overline{\tau})$
over the real point $0\in(\overline{C},\overline{\tau})$ has multiplicity
two. When endowed with its reduced structure, it is isomorphic to
the quotient of $\mathrm{Spec}(\mathbb{C}[x,y]/(xy+1))$ by the involution
$(x,y)\mapsto(-x,-y)$ hence to $\mathbb{A}_{*}^{1}=\mathrm{Spec}(\mathbb{C}[w^{\pm1}])$,
where $w=x^{2}$. The induced real structure is the composition of
the involution $w\mapsto w^{-1}$ with the complex conjugation. The
induced $\mathbb{G}_{m,\mathbb{C}}$-action is given by $t\cdot w=t^{2}w$.
Thus $(\overline{\pi}^{-1}(0)_{\mathrm{red}},\overline{\sigma}|_{\overline{\pi}^{-1}(0)_{\mathrm{red}}})$
is isomorphic to $\mathbb{S}^{1}$ on which $\mathbb{S}^{1}$-acts
with stabilizer $\mu_{2}$. 
\end{example}

\begin{lem}
\label{lem:Local-Reduction}Let $(S,\sigma)$ be a smooth real affine
surface with an effective $\mathbb{S}^{1}$-action $\mu:\mathbb{G}_{m,\mathbb{C}}\times S\rightarrow S$
determined by a regular real DPD-pair $(D,h)$ on a smooth real affine
curve $(C,\tau)$, and let $\pi:(S,\sigma)\rightarrow(C,\tau)$ be
the corresponding real quotient morphism. Then for every real point
$c$ of $(C,\tau)$ there exists a principal real affine open neighborhood
$(U,\tau|_{U})$ of $c$ and a regular real DPD-pair $(D',h')$ on
$(U,\tau|_{U})$ with the following properties:

1) $D'|_{U\setminus\{c\}}$ is the trivial divisor, $D'(c)\in[0,1[$
and $h'\in\Gamma(U,\mathcal{O}_{U})\cap\Gamma(U\setminus\{c\},\mathcal{O}_{U\setminus\{c\}}^{*})$. 

2) The surface $(\pi^{-1}(U),\sigma|_{\pi^{-1}(U)})$ is $\mathbb{S}^{1}$-equivariantly
isomorphic to that determined by the real DPD-pair $(D',h')$ on $(U,\tau|_{U})$. 

In particular, $\pi|_{\pi^{-1}(U\setminus\{c\})}:(\pi^{-1}(U\setminus\{c\}),\sigma|_{\pi^{-1}(U\setminus\{c\})})\rightarrow(U\setminus\{c\},\tau|_{U\setminus\{c\}})$
is an $\mathbb{S}^{1}$-torsor.
\end{lem}

\begin{proof}
Recall that by the construction described in $\S$ \ref{subsec:DPD-2-Surface},
we have 
\[
\Gamma(S,\mathcal{O}_{S})=A_{0}[D_{+},D]=\bigoplus_{m<0}\Gamma(C,\mathcal{O}_{C}(-mD_{-}))\oplus\Gamma(C,\mathcal{O}_{C})\oplus\bigoplus_{m>0}\Gamma(C,\mathcal{O}_{C}(mD_{+}))
\]
where $A_{0}=\Gamma(C,\mathcal{O}_{C})$, $D_{+}=D$ and $D_{-}=\tau^{*}D-\mathrm{div}(h)$.
Let $U=C_{g}$ be a real principal affine open neighborhood of $c$
for some real regular function $g$ on $(C,\tau)$, and let $(S|_{U},\sigma|_{U})=(\pi^{-1}(U),\sigma|_{\pi^{-1}(U)})$
be endowed with the induced $\mathbb{S}^{1}$-action. The graded coordinate
ring of $S|_{U}$ is isomorphic to the homogeneous localization 
\[
\Gamma(S,\mathcal{O}_{S})_{(g)}\cong\bigoplus_{m<0}\Gamma(U,\mathcal{O}_{C}(-mD_{-}))\oplus\Gamma(U,\mathcal{O}_{U})\oplus\bigoplus_{m>0}\Gamma(U,\mathcal{O}_{C}(mD_{+}))
\]
of $\Gamma(S,\mathcal{O}_{S})$ with respect to $g\in\Gamma(C,\mathcal{O}_{C})$.
It follows that $(S|_{U},\sigma|_{U})$ is $\mathbb{S}^{1}$-equivariantly
isomorphic to the real affine surface determined by the real DPD-pair
$(D|_{U},h|_{U})$ on the smooth real affine curve $(U,\tau|_{U})$.
For a small enough such real affine neighborhood $U$ of $c$, we
have $D(c')=0$ for every $c'\in U\setminus\{c\}$ and $h\in\Gamma(U\setminus\{c\},\mathcal{O}_{U\setminus\{c\}}^{*})$.
In particular, $D|_{U\setminus\{c\}}$ is a principal Cartier divisor
such that $D|_{U\setminus\{c\}}+\tau|_{U\setminus\{c\}}^{*}D|_{U\setminus\{c\}}=\mathrm{div}(h|_{U\setminus\{c\}})$,
which implies by Lemma \ref{lem:S1-torsors} that $\pi:(S|_{U\setminus\{c\}},\sigma|_{U\setminus\{c\}})\rightarrow(U\setminus\{c\},\tau|_{U\setminus\{c\}})$
is an $\mathbb{S}^{1}$-torsor. Shrinking $U$ further if necessary,
we can ensure in addition that $c=\mathrm{div}(f)$ for some real
regular function on $(U,\tau|_{U})$. Letting $\delta=\left\lfloor D(c)\right\rfloor $
be the round-down of the rational number $D(c)$, it follows from
Theorem \ref{thm:MainThm} 1) that $(S|_{U},\sigma|_{U})$ is $\mathbb{S}^{1}$-equivariantly
isomorphic to the surface determined by the regular real DPD-pair
\[
(D',h')=(D|_{U}-\mathrm{div}(f^{\delta}),(f^{-\delta}\tau^{*}f^{-\delta})h),
\]
on $(U,\tau|_{U})$. By construction, we have $D'=(D(c)-\delta)\{c\}$
where $D(c)-\delta\in[0,1[$ and $h'\in\Gamma(U\setminus\{c\},\mathcal{O}_{U\setminus\{c\}}^{*})$.
Since $(D',h')$ is a real DPD-pair, 
\[
\mathrm{ord}_{c}(h')\geq D'(c)+\tau^{*}(D')(c)=2D'(c)\geq0,
\]
which implies that $h'\in\Gamma(U,\mathcal{O}_{U})\cap\Gamma(U\setminus\{c\},\mathcal{O}_{U\setminus\{c\}}^{*})$. 
\end{proof}
\begin{defn}
Let $(S,\sigma)$ be a smooth real affine surface with an effective
$\mathbb{S}^{1}$-action $\mu:\mathbb{G}_{m,\mathbb{C}}\times S\rightarrow S$.
A $\mathbb{G}_{m,\mathbb{C}}$-orbit $Z$ is called \emph{principal}
if $Z$ endowed with the $\mathbb{G}_{m,\mathbb{C}}$-action induced
by $\mu$ is the trivial $\mathbb{G}_{m,\mathbb{C}}$-torsor. It is
called \emph{exceptional} otherwise. If $Z$ is in addition irreducible
and $\sigma$-invariant, we say that $(Z,\sigma|_{Z})$ is a principal
$\mathbb{S}^{1}$-orbit if $Z$ is a principal $\mathbb{G}_{m,\mathbb{C}}$-orbit,
and an exceptional $\mathbb{S}^{1}$-orbit otherwise. 
\end{defn}

\begin{thm}
\label{prop:Real-Except-Orb-1} Let $(S,\sigma)$ be a smooth real
affine surface with an effective $\mathbb{S}^{1}$-action $\mu:\mathbb{G}_{m,\mathbb{C}}\times S\rightarrow S$
determined by a regular real DPD-pair $(D,h)$ on a smooth real affine
curve $(C,\tau)$. Let $\pi:(S,\sigma)\rightarrow(C,\tau)$ be the
corresponding real quotient morphism and let $c$ be a real point
of $(C,\tau)$. Then exactly one of the following three possibilities
occurs:

\begin{enumerate}[label=\alph*)]

\item $D(c)\in\mathbb{Z}$ and $\mathrm{ord}_{c}(h)=2D(c)$. In this
case, there exists a real affine open neighborhood $(U,\tau|_{U})$
of $c$ such that $\pi|_{\pi^{-1}(U)}:(\pi^{-1}(U),\sigma|_{\pi^{-1}(U)})\rightarrow(U,\tau|_{U})$
is an $\mathbb{S}^{1}$-torsor. The fiber $(\pi^{-1}(c),\sigma|_{\pi^{-1}(c)})$
is an $\mathbb{S}^{1}$-torsor over $(c,\sigma_{\mathrm{Spec}(\mathbb{R})})$
which is either isomorphic to $\mathbb{S}^{1}$ if $\pi^{-1}(c)$
contains a real point of $(S,\sigma)$, or to the nontrivial $\mathbb{S}^{1}$-torsor
$\hat{\mathbb{S}}^{1}$ of Example \ref{exa:S1-torsor-pt} otherwise.
\\

\item $D(c)\in\frac{1}{2}\mathbb{Z}\setminus\mathbb{Z}$ and $\mathrm{ord}_{c}(h)=2D(c)$.
In this case, $(\pi^{-1}(c),\sigma|_{\pi^{-1}(c)})$ is a multiple
fiber of multiplicity $2$, whose reduction is an exceptional $\mathbb{S}^{1}$-orbit,
isomorphic to $\mathbb{S}^{1}$ on which $\mathbb{S}^{1}$ acts with
stabilizer $\mu_{2}$. \\

\item $D(c)\in\mathbb{Z}$ and $\mathrm{ord}_{c}(h)=2D(c)+1$. In
this case, the fiber $\pi^{-1}(c)$ is reduced, consisting of the
closures of two principal $\mathbb{G}_{m,\mathbb{C}}$ orbits $O^{+}$
and $O^{-}$ exchanged by the real structure $\sigma$, whose closures
$\overline{O}^{\pm}$ in $S$ are affine lines intersecting transversally
at an $\mathbb{S}^{1}$-fixed real point $p$ of $(S,\sigma)$. \\

\end{enumerate}

Furthermore, in cases b) and c), for every real affine open neighborhood
$(U,\tau|_{U})$ of $c\in C$, the restriction 
\[
\pi|_{\pi^{-1}(U\setminus\{c\})}:(\pi^{-1}(U\setminus\{c\}),\sigma|_{\pi^{-1}(U\setminus\{c\})})\rightarrow(U\setminus\{c\},\tau|_{U\setminus\{c\}})
\]
is a nontrivial $\mathbb{S}^{1}$-torsor.
\end{thm}

\begin{proof}
By Lemma \ref{lem:Local-Reduction}, there exists a real affine open
neighborhood $(U,\tau|_{U})$ of $c$ such that $c=\mathrm{div}(f)$
for some real regular function $f$ on $(U,\tau|_{U})$ and such that
$(S|_{U},\sigma|_{U})$ is $\mathbb{S}^{1}$-equivariantly isomorphic
over $(U,\tau|_{U})$ to the real affine surface determined by a regular
real DPD-pair $(D',h')$ such that $D'=D'(c)c$ where $D'(c)\in[0,1[$
and $h'\in\Gamma(U,\mathcal{O}_{U})\cap\Gamma(U\setminus\{c\},\mathcal{O}_{U\setminus\{c\}}^{*})$.
Since $D'(c)+\tau^{*}D'(c)\leq\mathrm{ord}_{c}(h')$ by definition
of a real DPD-pair, this leads to the following dichotomy:

I) If $2D'(c)=\mathrm{ord}_{c}(h')$ then since $D'(c)\in[0,1[$ and
$\mathrm{ord}_{c}(h')$ is an integer, we have either $D'(c)=0$ and
$\mathrm{ord}_{c}(h')=0$ or $D'(c)=\frac{1}{2}$ and $\mathrm{ord}_{c}(h')=1$.
In first case, $D'$ is the trivial divisor, and so is $\tau^{*}D'$.
Furthermore, since $h'$ does not vanish on $U$, $\pi|_{\pi^{-1}(U)}:(S|_{U},\sigma|_{U})=(\pi^{-1}(U),\sigma|_{\pi^{-1}(U)})\rightarrow(U,\tau|_{U})$
is an $\mathbb{S}^{1}$-torsor by Lemma \ref{lem:S1-torsors}. By
Theorem \ref{thm:MainThm} 1) and Example \ref{exa:S1-torsor-pt},
$(\pi^{-1}(c),\sigma|_{\pi^{-1}(c)})$ is isomorphic either to $\mathbb{S}^{1}$
if $h_{c}(c)\in\mathbb{R}_{>0}$, or to $\hat{\mathbb{S}}^{1}$ otherwise.
This yields case a). 

In the second case, we have $D'=\frac{1}{2}\{c\}$ and it follows
from \cite[Theorem 18 (a)]{FZ03} that $\pi^{-1}(c)=2Z$ where $Z$
is an exceptional $\mathbb{G}_{m,\mathbb{C}}$-orbit isomorphic to
a punctured affine line on which $\mathbb{G}_{m,\mathbb{C}}$ acts
with stabilizer $\mu_{2}$. The real curve $(Z,\sigma|_{Z})$ endowed
with the restriction of $\mu$ is thus isomorphic either to $\mathbb{S}^{1}$
if it contains a real point or to $\hat{\mathbb{S}}^{1}$ otherwise,
on which $\mathbb{S}^{1}$ acts with stabilizer $\mu_{2}$. We claim
that the case where $(Z,\sigma|_{Z})$ is isomorphic to $\hat{\mathbb{S}}^{1}$
does not occur. Indeed, the real structure $\tau|_{U}$ on $U$ lifts
in a unique way to a real structure $\tilde{\tau}$ on $\tilde{U}=\mathrm{Spec}(\Gamma(U,\mathcal{O}_{U})[X]/(X^{2}-f))$
such that $\tilde{\tau}^{*}X=X$, for which the morphism $\psi:(\tilde{U},\tilde{\tau})\rightarrow(U,\tau|_{U})$
the induced morphism $\psi:(\tilde{U},\tilde{\tau})\rightarrow(U,\tau|_{U})$
is a real double cover totally ramified over $c$ and \'etale elsewhere.
The normalization over the fiber product $S\times_{U}\tilde{U}$ is
a smooth real affine surface $(\tilde{S},\tilde{\sigma})$ and the
action of the Galois group $\mathbb{Z}_{2}$ of the cover $\psi:(\tilde{U},\tilde{\tau})\rightarrow(U,\tau|_{U})$
lifts to a free real $\mathbb{Z}_{2}$-action on $(\tilde{S},\tilde{\sigma})$
for which we have $(S|_{U},\sigma|_{U})\cong(\tilde{S},\tilde{\sigma})/\mathbb{Z}_{2}$,
the quotient morphism $\Psi:(\tilde{S},\tilde{\sigma})\rightarrow(S|_{U},\sigma|_{U})\cong(\tilde{S},\tilde{\sigma})/\mathbb{Z}_{2}$
being \'etale. The $\mathbb{S}^{1}$-action $\mu$ on $S$ lifts
to an effective $\mathbb{S}^{1}$-action $\tilde{\mu}$ on $(\tilde{S},\tilde{\sigma})$,
whose real quotient morphism $\tilde{\pi}:(\tilde{S},\tilde{\sigma})\rightarrow(\tilde{U},\tilde{\tau})$
is equal to the composition of the normalization morphism $\nu:\tilde{S}\rightarrow S\times_{U}\tilde{U}$
with the projection $\mathrm{pr}_{\tilde{U}}$. Letting $\tilde{c}=\psi^{-1}(c)$,
$\tilde{\pi}^{-1}(\tilde{c})$ is reduced and $(\tilde{\pi}^{-1}(\tilde{c}),\tilde{\sigma}|_{\tilde{\pi}^{-1}(\tilde{c})})$
is an $\mathbb{S}^{1}$-torsor over $(\tilde{c},\sigma_{\mathrm{Spec}(\mathbb{R})})$.
By Example \ref{exa:S1-torsor-pt}, it is isomorphic to $\mathbb{A}_{\mathbb{C}}^{1}\setminus\{0\}=\mathrm{Spec}(\mathbb{C}[u^{\pm1}])$
endowed with a real structure given as the composition of the complex
conjugation either with the involution $u\mapsto u^{-1}$ or with
the involution $u\mapsto-u^{-1}$. Furthermore, the $\mathbb{Z}_{2}$-action
on $\tilde{S}$ restricts to a $\mathbb{G}_{m,\mathbb{C}}$-equivariant
free $\mathbb{Z}_{2}$-action on $\tilde{\pi}^{-1}(\tilde{c})\cong\mathrm{Spec}(\mathbb{C}[u^{\pm1}])$
compatible with the real structure $\tilde{\sigma}|_{\tilde{\pi}^{-1}(\tilde{c})}$.
The latter is thus necessarily given by $u\mapsto-u$, and the quotient
morphism $\Psi$ restricts on $(\tilde{\pi}^{-1}(\tilde{c}),\tilde{\sigma}|_{\tilde{\pi}^{-1}(\tilde{c})})$
to an \'etale double cover $(\tilde{\pi}^{-1}(c),\tilde{\sigma}|_{\tilde{\pi}^{-1}(\tilde{c})})\rightarrow(Z,\sigma|_{Z})$.
We conclude that $Z\cong\mathrm{Spec}(\mathbb{C}[w^{\pm1}])$, where
$w=u^{2}$ and that $\sigma|_{Z}$ is the real structure given as
the composition of the involution $w\mapsto w^{-1}$ with the complex
conjugation, which shows that $(Z,\sigma|_{Z})$ is isomorphic to
$\mathbb{S}^{1}$. Finally, since $\mathrm{ord}_{c}(h')=1$, there
cannot exist any rational function on $U$ such that $h'=g\tau^{*}g$.
It follows that for every real affine open neighborhood $(V,\tau|_{V})$
of $c$ contained in $(U,\tau|_{U})$, the restriction of $\pi$ over
$V\setminus\{c\}$ is a nontrivial $\mathbb{S}^{1}$-torsor. This
yields case b). 

II) Otherwise, if $2D'(c)-\mathrm{ord}_{c}(h)<0$, then since by hypothesis
$(D,h)$ whence $(D',h')$ is a regular real DPD-pair, the rational
numbers $D'(c)$ and $D'(\tau(c))-\mathrm{ord}_{c}(h')=D'(c)-\mathrm{ord}_{c}(h')$
form a regular pair. Since $D'(c)\in[0,1[$ and $\mathrm{ord}_{c}(h')>2D'(c)$
is an integer, the only possibility is that $D'(c)=0$ and $\mathrm{ord}_{c}(h')=1$.
By \cite[Theorem 18 (b)]{FZ03}, the fiber $\pi^{-1}(c)$ is then
reduced consisting of the closure of two principal $\mathbb{G}_{m,\mathbb{C}}$-orbits
$O^{+}$ and $O^{-}$ whose closures $\overline{O}^{\pm}$ in $S$
are affine lines intersecting transversally at a $\mathbb{G}_{m,\mathbb{C}}$-fixed
point $p\in\pi^{-1}(c)$. The defining ideals of $\overline{O}^{+}$
and $\overline{O}^{-}$ in the graded coordinate ring $\Gamma(S|_{U},\mathcal{O}_{S})\otimes_{\Gamma(U,\mathcal{O}_{U})}(\Gamma(U,\mathcal{O}_{U})/f)$
of the scheme theoretic fiber $\pi^{-1}(c)$ are the positive and
negative part respectively. Since by Lemma \ref{lem:Hyperbolic-decomposition},
$\sigma^{*}$ exchanges the positive and negative parts of the grading
of $\Gamma(S|_{U},\mathcal{O}_{S})$, it follows that $\sigma$ exchanges
$\overline{O}^{+}$ and $\overline{O}^{-}$, hence that $p$ is a
$\sigma$-invariant point. As in the previous case, the fact that
$\mathrm{ord}_{c}(h')=1$ implies that for every real affine open
neighborhood $(V,\tau|_{V})$ of $c$ contained in $(U,\tau|_{U})$,
the restriction of $\pi$ over $V\setminus\{c\}$ is a nontrivial
$\mathbb{S}^{1}$-torsor. This yields case c).
\end{proof}
Since the only proper algebraic subgroups of $\mathbb{G}_{m,\mathbb{C}}$
are cyclic groups, the exceptional orbits of a $\mathbb{G}_{m,\mathbb{C}}$-action
are either $\mathbb{G}_{m,\mathbb{C}}$-fixed points or closed curves
isomorphic to the punctured affine line $\mathbb{A}_{\mathbb{C}}^{1}\setminus\{0\}$
on which $\mathbb{G}_{m,\mathbb{C}}$ acts with stabilizer isomorphic
to a cyclic group $\mu_{m}$ of order $m\geq2$. While there exist
smooth complex affine surfaces $S$ endowed with hyperbolic $\mathbb{G}_{m,\mathbb{C}}$-actions
admitting $1$-dimensional exceptional orbits with stabilizers $\mu_{m}$
for every $m\geq2$, for instance $(\mathbb{A}_{\mathbb{C}}^{1}\setminus\{0\})\times\mathbb{A}_{\mathbb{C}}^{1}=\mathrm{Spec}(\mathbb{C}[x^{\pm1},y])$
endowed with the $\mathbb{G}_{m,\mathbb{C}}$-action $t\cdot(x,y)=(t^{m}x,t^{-1}y)$,
the possible types of exceptional $\mathbb{S}^{1}$-orbits on smooth
real affine surfaces with effective $\mathbb{S}^{1}$-actions are
much more restricted: 
\begin{cor}
The exceptional $\mathbb{S}^{1}$-orbits on a smooth real affine surface
$(S,\sigma)$ with an effective $\mathbb{S}^{1}$-action are either
real $\mathbb{S}^{1}$-fixed points or closed curves isomorphic to
$\mathbb{S}^{1}$ on which $\mathbb{S}^{1}$ acts with stabilizer
$\mu_{2}$.
\end{cor}

\begin{proof}
Let $\pi:(S,\sigma)\rightarrow(C,\tau)$ be the real quotient morphism
for the given $\mathbb{S}^{1}$-action. The image by $\pi$ of a real
exceptional $\mathbb{S}^{1}$-orbit $(Z,\sigma|_{Z})$ is a $\tau$-invariant
proper closed subset of $C$, which is irreducible since $Z$ is irreducible.
So $\pi(Z)$ is a real point $c$ of $(C,\tau)$. Since $Z$ is an
exceptional $\mathbb{G}_{m,\mathbb{C}}$-orbit, the assertion follows
from Theorem \ref{prop:Real-Except-Orb-1}, cases b) and c). 
\end{proof}
The following proposition records the possible structures of fibers
of the real quotient morphism $\pi:(S,\sigma)\rightarrow(C,\tau)$
over pairs of non-real complex points $q$ and $\tau(q)$ of $C$.
Its proof, which is similar to that of Theorem \ref{prop:Real-Except-Orb-1},
is left to the reader. 
\begin{prop}
\label{rem:Fibers-pairs-conjugate-points} Let $(S,\sigma)$ be the
smooth real affine surface with effective $\mathbb{S}^{1}$-action
$\mu:\mathbb{G}_{m,\mathbb{C}}\times S\rightarrow S$ determined by
a regular real DPD-pair $(D,h)$ on a smooth real affine curve $(C,\tau)$.
Let $\pi:(S,\sigma)\rightarrow(C,\tau)$ be its real quotient morphism,
and let $q$ and $\tau(q)$ be a pair of non-real complex points of
$C$ exchanged by the real structure $\tau$. Then exactly one of
the following possibilities occurs:

\begin{enumerate}

\item $D(q)+D(\tau(q))=\mathrm{ord}_{q}(h)=\mathrm{ord}_{\tau(q)}(h)$
and: 

a) Either $D(q)$ and $D(\tau(q))$ both belong to $\mathbb{Z}$ and
then $\pi^{-1}(q)$ and $\pi^{-1}(\tau(q))$ are principal $\mathbb{G}_{m,\mathbb{C}}$-orbits.
Furthermore, there exists a real affine open neighborhood $(U,\tau|_{U})$
of $q\cup\tau(q)$ such that $\pi|_{\pi^{-1}(U)}:(\pi^{-1}(U),\sigma|_{\pi^{-1}(U)})\rightarrow(U,\tau|_{U})$
is an $\mathbb{S}^{1}$-torsor.

b) Or $D(q)$ and $D(\tau(q))$ both belong to $\mathbb{Q}\setminus\mathbb{Z}$
and then $\pi^{-1}(q)$ and $\pi^{-1}(\tau(q))=\sigma(\pi^{-1}(q))$
are $1$-dimensional exceptional $\mathbb{G}_{m,\mathbb{C}}$-orbits
of multiplicity $m\geq2$ on which $\mathbb{G}_{m,\mathbb{C}}$ acts
with stabilizer $\mu_{m}$. \\

\item$D(q)+D(\tau(q))<\mathrm{ord}_{q}(h)=\mathrm{ord}_{\tau(q)}(h)$.
Then $\pi^{-1}(q)_{\mathrm{red}}=\overline{O}_{q}^{+}\cup\overline{O}_{q}^{-}$,
where $O_{q}^{+}$ and $O_{q}^{-}$ are 1-dimensional $\mathbb{G}_{m,\mathbb{C}}$-orbits
whose closures $\overline{O}_{q}^{\pm}$ in $S$ are affine lines
intersecting transversally at a $\mathbb{G}_{m,\mathbb{C}}$-fixed
point $p$. Furthermore, the fiber $\pi^{-1}(\tau(q))_{\mathrm{red}}=\sigma(\pi^{-1}(q)_{\mathrm{red}})$
is equal to 
\[
\pi^{-1}(\tau(q))_{\mathrm{red}}=\overline{O}_{\tau(q)}^{+}\cup\overline{O}_{\tau(q)}^{-}=\sigma(\overline{O}_{q}^{-})\cup\sigma(\overline{O}_{q}^{+}).
\]

\end{enumerate}
\end{prop}

\begin{example}
Let $(S_{\varepsilon},\sigma)$, where $\varepsilon=\pm1$, be the
smooth real affine surface with equation $xy=\varepsilon(z^{2}+1)$
in $\mathbb{A}_{\mathbb{C}}^{3}=\mathrm{Spec}(\mathbb{C}[x,y,z])$
endowed with the real structure given by the composition of the involution
$(x,y,z)\mapsto(y,x,z)$ with the complex conjugation. The $\mathbb{G}_{m,\mathbb{C}}$-action
$\mu$ on $S$ given by $t\cdot(x,y,z)=(tx,t^{-1}y,z)$ defines a
real action of $\mathbb{S}^{1}$ on $(S_{\varepsilon},\sigma)$. The
categorical quotient for the $\mathbb{G}_{m,\mathbb{C}}$-action is
the affine line $C=\mathrm{Spec}(A_{0})$, where $A_{0}=\mathbb{C}[xy,z]/(xy-\varepsilon(z^{2}+1))\cong\mathbb{C}[z]$,
and the quotient morphism $\pi:S_{\varepsilon}\rightarrow C$ is a
real morphism for the real structures $\sigma$ and $\tau=\sigma_{\mathbb{A}_{\mathbb{R}}^{1}}$
on $S_{\varepsilon}$ and $C$ respectively. The decomposition of
the coordinate ring $A_{\varepsilon}$ of $S_{\varepsilon}$ into
semi-invariant subspaces if given by 
\begin{align*}
A_{\varepsilon} & =\bigoplus_{m\in\mathbb{Z}}A_{\varepsilon,m}=\bigoplus_{m<0}A_{0}\cdot y^{-m}\oplus A_{0}\oplus\bigoplus_{m>0}A_{0}\cdot x^{m}\\
 & =\bigoplus_{m<0}A_{0}\cdot(xy)^{-m}x^{m}\oplus A_{0}\oplus\bigoplus_{m>0}A_{0}\cdot x^{m}\\
 & =\bigoplus_{m<0}A_{0}\cdot(\varepsilon(z^{2}+1))^{-m}s^{m}\oplus A_{0}\oplus\bigoplus_{m>0}A_{0}\cdot s^{m}
\end{align*}
where $s=x$. A corresponding real DPD-pair on $(C,\tau)$ is thus
given by $(D,h)=(0,x\sigma^{*}x)=(0,\varepsilon(z^{2}+1))$. Noting
that $1+z^{2}=(1+iz)(1-iz)=f\tau^{*}f$, we deduce from Theorem \ref{thm:MainThm}
1) that $(S_{\varepsilon},\sigma)$ is also given by the real DPD-pair
$(D',h')=(D-\mathrm{div}(f),\varepsilon)=(1\cdot\{i\},\varepsilon)$. 

It then follows from Lemma \ref{lem:S1-torsors} and Example \ref{exa:S1-torsor-pt}
that the restriction of $\pi:(S_{\varepsilon},\sigma)\rightarrow(C,\tau)$
over the real affine open subset $U=C\setminus\{\pm i\}$ is either
the trivial $\mathbb{S}^{1}$-torsor $(U,\tau|_{U})\times\mathbb{S}^{1}$
if $\varepsilon=1$, or the $\mathbb{S}^{1}$-torsor $(U,\tau|_{U})\times\hat{\mathbb{S}}^{1}$
if $\varepsilon=-1$. On the other hand, $\pi^{-1}(\{\pm i\})\cong\mathrm{Spec}(\mathbb{C}[x,y]/(xy))$
is the union of two copies $\overline{O}_{\pm i}^{+}=\{x=z\mp i=0\}$
and $\overline{O}_{\pm i}^{-}=\{y=z\mp i=0\}$ of the complex affine
line intersecting at the point $\{x=y=z\mp i=0\}$, and since $\sigma^{*}x=y$,
we have $\sigma(\overline{O}_{\pm i}^{\pm})=\overline{O}_{\mp i}^{\mp}$. 
\end{example}

\section{Rational real algebraic models of compact differential surfaces with
circle actions }

This section is devoted to the proof of Theorem \ref{thm:MainThm-Model}.
We first construct explicit rational projective and affine real algebraic
models of compact real manifolds of dimension $2$ without boundary
endowed with effective differentiable $S^{1}$-actions. Then we show
that each rational quasi-projective real algebraic model of such a
manifold is $\mathbb{S}^{1}$-equivariantly birationally diffeomorphic
to one of these models. 

\subsection{\label{subsec:Rational-affine-models-Construct}Rational affine models
with compact real loci}

It is a classical result (see e.g. \cite[I.3.a]{Au04}) that a compact
connected real manifold of dimension $2$ without boundary endowed
with an effective differentiable $S^{1}$-action is equivariantly
diffeomorphic to one of the following manifolds: the torus $T=S^{1}\times S^{1}$,
the sphere $S^{2}$, the projective plane $\mathbb{RP}^{2}$ and the
Klein bottle $K$. We now describe smooth rational projective and
affine real algebraic models with $\mathbb{S}^{1}$-actions of these
compact differential surfaces. 

\subsubsection{\label{subsec:Torus-Model}Equivariant rational models of the torus}

The group $S^{1}$ acts on the torus $T=S^{1}\times S^{1}$ by translations
on the second factor. All the orbits are principal, and the orbit
space is equal to $S^{1}$. 

A rational projective model of $T$ is the complexification $(\mathbb{P}_{\mathbb{C}}^{1}\times\mathbb{P}_{\mathbb{C}}^{1},\sigma_{\mathbb{P}_{\mathbb{R}}^{1}\times\mathbb{P}_{\mathbb{R}}^{1}})$
of $\mathbb{P}_{\mathbb{R}}^{1}\times\mathbb{P}_{\mathbb{R}}^{1}$
on which $\mathbb{S}^{1}$ acts on the second factor via the projective
representation induced by the representation $\mathbb{S}^{1}\rightarrow SO_{2}(\mathbb{R})$
in Definition \ref{def:S1-action}. The action of $\mathbb{S}^{1}$
on $(\mathbb{P}_{\mathbb{C}}^{1}=\mathrm{Proj}(\mathbb{C}[u,v]),\sigma_{\mathbb{P}_{\mathbb{R}}^{1}})$
induced by the representation $\mathbb{S}^{1}\rightarrow SO_{2}(\mathbb{R})$
has a pair of non-real fixed points $[1:i]$ and $[1:-i]$ exchanged
by the real structure $\sigma_{\mathbb{P}_{\mathbb{R}}^{1}}$. Their
complement is isomorphic to the trivial $\mathbb{S}^{1}$-torsor.
The affine open subset 
\[
S_{1}=(\mathbb{P}_{\mathbb{C}}^{1}\times\mathbb{P}_{\mathbb{C}}^{1},\sigma_{\mathbb{P}_{\mathbb{R}}^{1}\times\mathbb{P}_{\mathbb{R}}^{1}})\setminus\{[1:\pm i]\times\mathbb{P}_{\mathbb{C}}^{1}\cup\mathbb{P}_{\mathbb{C}}^{1}\times[1:\pm i]\}
\]
is $\sigma_{\mathbb{P}_{\mathbb{R}}^{1}\times\mathbb{P}_{\mathbb{R}}^{1}}$-invariant
and $\mathbb{S}^{1}$-invariant. Furthermore, letting $\sigma_{1}$
be the restriction of $\sigma_{\mathbb{P}_{\mathbb{R}}^{1}\times\mathbb{P}_{\mathbb{R}}^{1}}$
to $S_{1}$, the inclusion $(S_{1},\sigma_{1})\hookrightarrow(\mathbb{P}_{\mathbb{C}}^{1}\times\mathbb{P}_{\mathbb{C}}^{1},\sigma_{\mathbb{P}_{\mathbb{R}}^{1}\times\mathbb{P}_{\mathbb{R}}^{1}})$
is an $\mathbb{S}^{1}$-equivariant birational diffeomorphism. It
follows that $(S_{1},\sigma_{1})$ is a rational affine model of $T$,
equivariantly isomorphic the product $(Q_{1,\mathbb{C}},\sigma_{Q_{1}})\times\mathbb{S}^{1}$
of $\mathbb{S}^{1}$ with the complexification of the smooth affine
quadric curve $Q_{1}\subset\mathrm{Spec}(\mathbb{R}[u,v])$ with equation
$u^{2}+v^{2}=1$, on which $\mathbb{S}^{1}$ acts by translations
on the second factor. 

The projection 
\[
\pi_{1}=\mathrm{pr}_{Q_{1,\mathbb{C}}}:(S_{1},\sigma_{1})=(Q_{1,\mathbb{C}},\sigma_{Q_{1}})\times\mathbb{S}^{1}\rightarrow(C_{1},\tau_{1})=(Q_{1,\mathbb{C}},\sigma_{Q_{1}})
\]
is the trivial $\mathbb{S}^{1}$-torsor. A corresponding real DPD-pair
on $(C_{1},\tau_{1})$ is $(D_{1},h_{1})=(0,1)$, where $0$ denotes
the trivial Weil divisor. The image by $\pi_{1}$ of the real locus
of $(S_{1},\sigma_{1})$ is equal to the real locus $S^{1}$ of $(C_{1},\tau_{1})$. 

\subsubsection{\label{subsec:ModelSphere} Equivariant rational models of the sphere }

The group $S^{1}$ acts on the unit sphere $S^{2}$ in $\mathbb{R}^{3}$
by rotations around a fixed axis. All the orbits are principal, except
for the two fixed points where the axis meets the sphere, and the
orbit space is a closed interval, each of its ends corresponding to
a non-principal orbit. 

A rational projective model is given by the complexification of the
smooth quadric 
\[
Q=\{u^{2}+v^{2}+z^{2}-w^{2}=0\}\subset\mathbb{P}_{\mathbb{R}}^{3}=\mathrm{Proj}_{\mathbb{R}}(\mathbb{R}[u,v,z,w])
\]
endowed with the restriction of the $\mathbb{S}^{1}$-action on $(\mathbb{P}_{\mathbb{C}}^{3},\sigma_{\mathbb{P}_{\mathbb{R}}^{3}})$
defined by the projective representation induced by the direct sum
of the representation $\mathbb{S}^{1}\rightarrow SO_{2}(\mathbb{R})$
with the trivial $2$-dimensional representation. 

The $\mathbb{S}^{1}$-invariant real hyperplane section $H=\{w=0\}$
of $(Q_{\mathbb{C}},\sigma_{Q})$ has empty real locus. Its complement
is $\mathbb{S}^{1}$-equivariantly isomorphic to the complexification
$(S_{2},\sigma_{2})=(\mathbb{S}_{\mathbb{C}}^{2},\sigma_{\mathbb{S}^{2}})$
of the smooth affine quadric $\mathbb{S}^{2}$ in $\mathrm{Spec}(\mathbb{R}[u,v,z])$
defined by the equation $u^{2}+v^{2}+z^{2}=1$, on which $\mathbb{S}^{1}$
acts by the restriction of the direct sum of the representation $\mathbb{S}^{1}\rightarrow SO_{2}(\mathbb{R})$
with the trivial $1$-dimensional representation. By construction,
the inclusion $(S_{2},\sigma_{2})\hookrightarrow(Q_{\mathbb{C}},\sigma_{Q})$
is an $\mathbb{S}^{1}$-equivariant birational diffeomorphism. 

The real quotient morphism of the $\mathbb{S}^{1}$-action on $(S_{2},\sigma_{2})$
is the projection 
\[
\pi_{2}=\mathrm{pr}_{z}:(S_{2},\sigma_{2})\rightarrow(C_{2},\tau_{2})=(\mathbb{A}_{\mathbb{C}}^{1}=\mathrm{Spec}(\mathbb{C}[z],\sigma_{\mathbb{A}_{\mathbb{R}}^{1}})
\]
and a corresponding real DPD-pair on $(C_{2},\tau_{2})$ is $(D_{2},h_{2})=(0,1-z^{2})$,
where $0$ denotes the trivial Weil divisor. The image by $\pi_{2}$
of the real locus of $(S_{2},\sigma_{2})$ is the segment $[-1,1]$
of the real locus $\mathbb{R}$ of $(\mathbb{A}_{\mathbb{C}}^{1},\sigma_{\mathbb{A}_{\mathbb{R}}^{1}})$.
The restriction of $\pi_{2}$ over the principal real afine open subset
$(\mathbb{A}_{\mathbb{C}}^{1}\setminus\{\pm1\},\sigma_{\mathbb{A}_{\mathbb{R}}^{1}}|_{\mathbb{A}_{\mathbb{C}}^{1}\setminus\{\pm1\}})$
is a nontrivial $\mathbb{S}^{1}$-torsor. The fibers of $\pi_{2}$
over the real points $\pm1$ of $(C_{2},\tau_{2})$ are of type c)
in Theorem \ref{prop:Real-Except-Orb-1}. Their respective real loci
consist of a unique point $p_{\pm}=(0,0,\pm1)$, which is an $\mathbb{S}^{1}$-fixed
point.\\

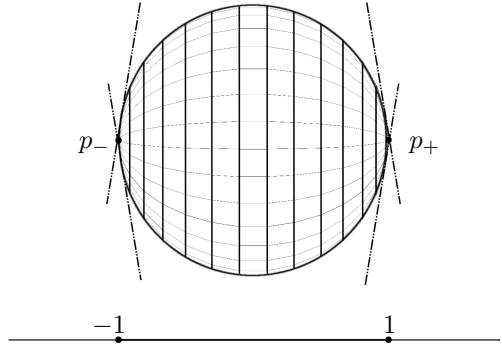
\begin{figure}[h!]

\includegraphics[scale=0.20]{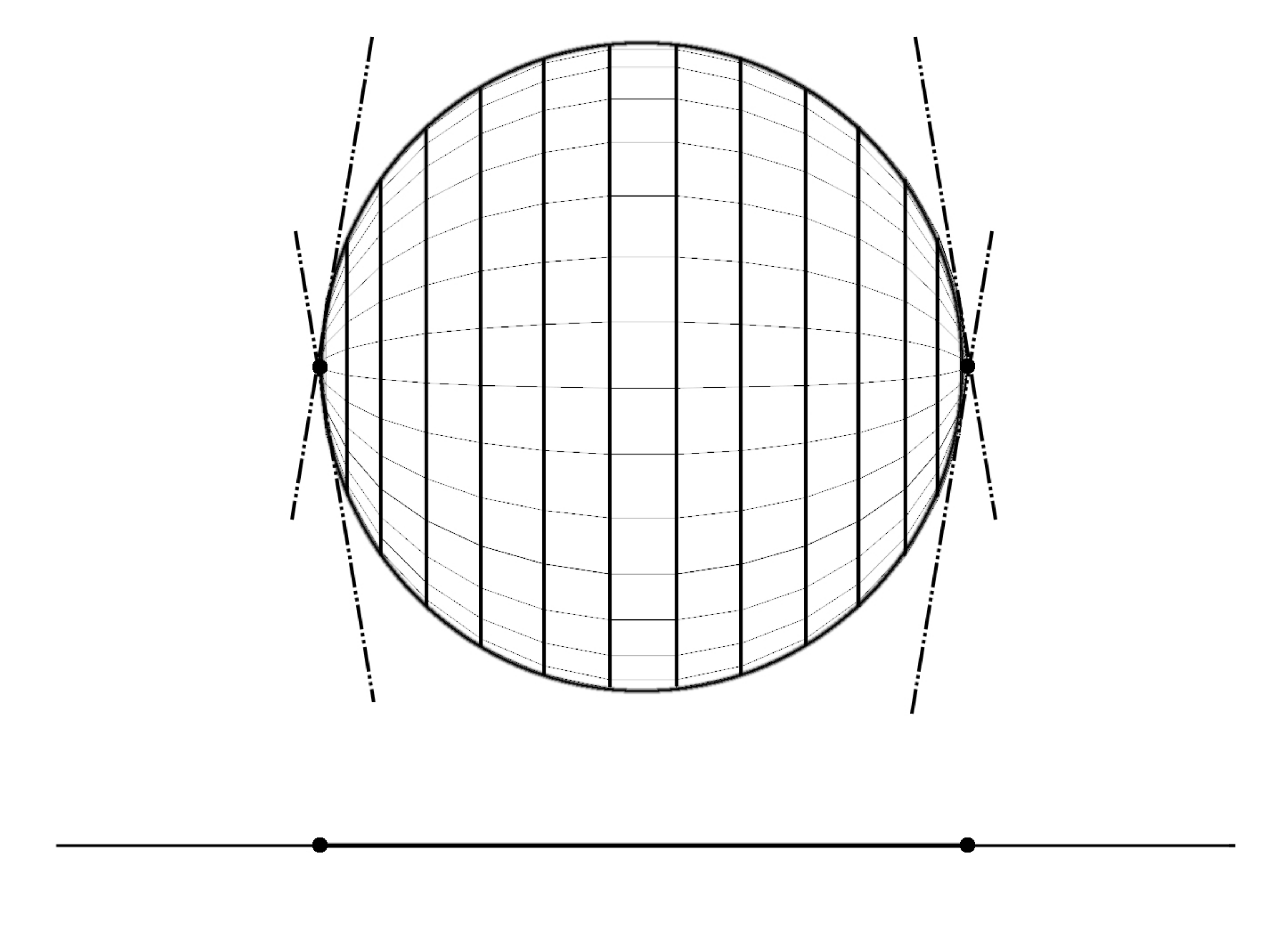}
\begin{picture}(0,0)
\put(-45,85){$p_+$}
\put(-170,85){$p_-$}
\put(-55,15){$1$}
\put(-165,15){$-1$}
\end{picture}

\caption{Projection of the real locus $S^2$ of $(S_2,\sigma_2)$ onto the interval $[-1,1]$. The dashed lines represent the closures $\overline{O}^{\pm}$ of the two principal $\mathbb{G}_{m,\mathbb{C}}$-orbits in $\pi_2^{-1}(\pm 1)$ exchanged by $\sigma_2$ and intersecting at the unique $\mathbb{S}^1$-fixed point $p_{\pm}\subset\pi_2^{-1}(\pm 1)$.}
\label{SphereMod}
\end{figure}

\subsubsection{Equivariant rational models of the projective plane}

The group $S^{1}$ acts on $\mathbb{RP}^{2}$, viewed as the projective
compactification of $\mathbb{R}^{2}$ by adding a ``line at infinity''
$\mathbb{RP}^{1}\cong S^{1}$, by the extension of the linear action
of $S^{1}=SO(2)$ on $\mathbb{R}^{2}$ to an action on $\mathbb{RP}^{2}$
leaving the line at infinity invariant. All the orbits are principal,
except for two of them: one is a fixed point corresponding to the
origin of $\mathbb{R}^{2}$ and the other is the line at in infinity,
equivariantly isomorphic to $S^{1}$ on which $S^{1}$ act with stabilizer
$\mu_{2}$. An $S^{1}$-invariant tubular neighborhood of this second
non principal orbit is isomorphic to the quotient $(S^{1}\times\mathbb{R})/(z,u)\sim(-z,-u)$,
that is, to an open Moebius band $B$, endowed with the induced $S^{1}$-action.
The orbit space of this $S^{1}$-action on $\mathbb{RP}^{2}$ is a
closed interval, each of its ends corresponding to a non-principal
orbit. 

A rational projective model is the complexification $(\mathbb{P}_{\mathbb{C}}^{2}=\mathrm{Proj}(\mathbb{C}[u,v,z]),\sigma_{\mathbb{P}_{\mathbb{R}}^{2}})$
of $\mathbb{P}_{\mathbb{R}}^{2}$ endowed with the $\mathbb{S}^{1}$-action
defined by the projective representation induced by the direct sum
of the representation $\mathbb{S}^{1}\rightarrow SO_{2}(\mathbb{R})$
with the trivial $1$-dimensional representation. The smooth conic
$\Delta$ in $\mathbb{P}_{\mathbb{C}}^{2}$ with equation $u^{2}+v^{2}+z^{2}=0$
is $\sigma_{\mathbb{P}_{\mathbb{R}}^{2}}$-invariant and $\mathbb{S}^{1}$-invariant
and has empty real locus. Its complement $S_{3}=\mathbb{P}_{\mathbb{C}}^{2}\setminus\Delta$
endowed with the restriction $\sigma_{3}$ of $\sigma_{\mathbb{P}_{\mathbb{R}}^{2}}$
and the induced $\mathbb{S}^{1}$-action is thus an affine model of
$\mathbb{RP}^{2}$. By construction, the inclusion $(S_{3},\sigma_{3})\hookrightarrow(\mathbb{P}_{\mathbb{C}}^{2},\sigma_{\mathbb{P}_{\mathbb{R}}^{2}})$
is an $\mathbb{S}^{1}$-equivariant birational diffeomorphism. \\

Another rational projective model of $\mathbb{RP}^{2}$ with $\mathbb{S}^{1}$-action
is obtained by blowing-up the smooth quadric 
\[
Q_{\mathbb{C}}=\{u^{2}+v^{2}+z^{2}-w^{2}=0\}\subset\mathbb{P}_{\mathbb{C}}^{3},
\]
endowed with the real structure $\sigma_{Q}$ and the $\mathbb{S}^{1}$-action
defined in subsection \ref{subsec:ModelSphere}, at the real $\mathbb{S}^{1}$-fixed
point $p_{-}=\left[0:0:-1:1\right]$. Letting $\alpha:(\tilde{Q}_{\mathbb{C}},\sigma_{\tilde{Q}})\rightarrow(Q_{\mathbb{C}},\sigma_{Q})$
be the blow-up morphism, with exceptional divisor $E_{-}\cong(\mathbb{P}_{\mathbb{C}}^{1},\sigma_{\mathbb{P}_{\mathbb{R}}^{1}})$,
the $\mathbb{S}^{1}$-action on $(Q_{\mathbb{C}},\sigma_{Q})$ lifts
to an action on $(\tilde{Q}_{\mathbb{C}},\sigma_{\tilde{Q}})$ for
which the real locus of $(\tilde{Q}_{\mathbb{C}},\sigma_{\tilde{Q}})$
endowed with the induced $S^{1}$-action is equivariantly diffeomorphic
to the $S^{1}$-equivariant connected sum $\mathbb{RP}^{2}\sharp_{S^{1}}S^{2}\simeq\mathbb{RP}^{2}$
endowed with the $S^{1}$-action defined above. The proper transforms
in $(\tilde{Q}_{\mathbb{C}},\sigma_{\tilde{Q}})$ of the curves $\ell_{-}=\{u+iv=z+w=0\}$
and $\overline{\ell}_{-}=\{u-iv=z+w=0\}$ in $(Q_{\mathbb{C}},\sigma_{Q})$
are a pair of non-real disjoint smooth rational curves with self-intersection
$-1$, exchanged by the real structure $\sigma_{\tilde{Q}}$. Their
union is an $\mathbb{S}^{1}$-invariant real closed subset $F_{-}$
of $(\tilde{Q}_{\mathbb{C}},\sigma_{\tilde{Q}})$ and the contraction
of $F_{-}$ is an $\mathbb{S}^{1}$-equivariant birational diffeomorphism
$\alpha':(\tilde{Q}_{\mathbb{C}},\sigma_{\tilde{Q}})\rightarrow(\mathbb{P}_{\mathbb{C}}^{2},\sigma_{\mathbb{P}_{\mathbb{R}}^{2}})$
which maps the proper transform $\tilde{H}\subset\tilde{Q}_{\mathbb{C}}$
of the curve $H=\left\{ w=0\right\} \subset Q_{\mathbb{C}}$ onto
the smooth conic $\Delta\subset\mathbb{P}_{\mathbb{C}}^{2}$ (see
Figure \ref{QuadricP2}). 

\begin{figure}[h!]

\includegraphics[scale=0.50]{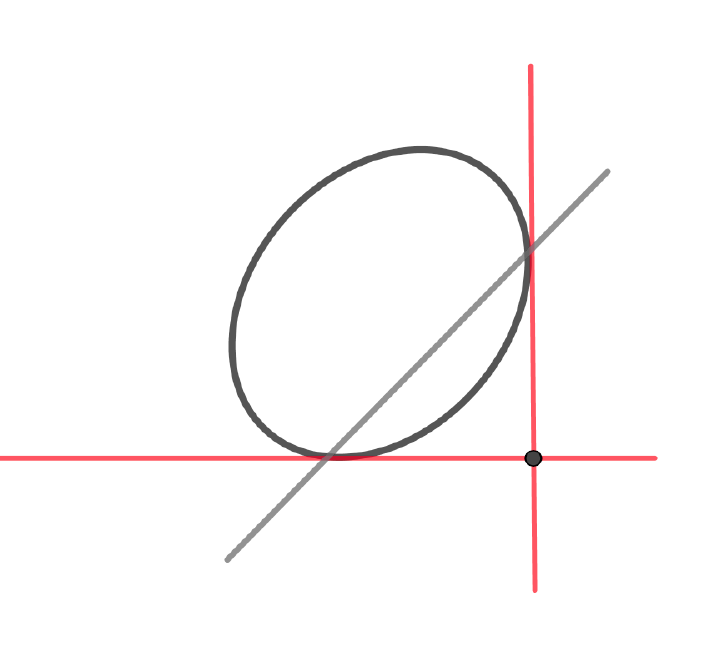} \hspace{0.5cm} 
\includegraphics[scale=0.50]{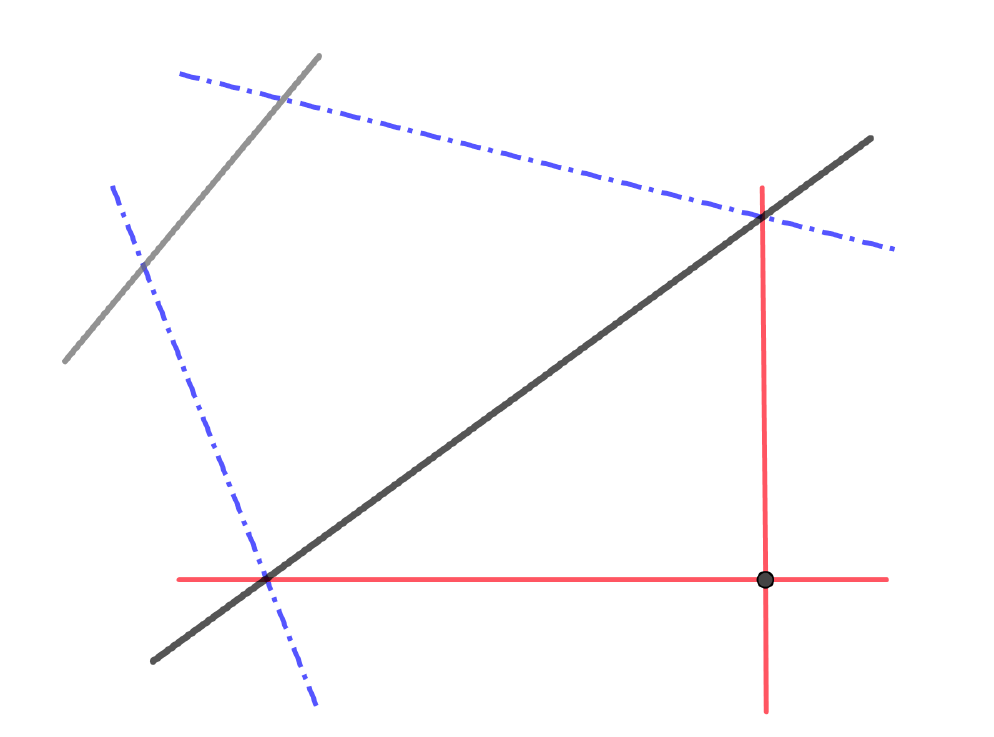} \hspace{0.5cm}
\includegraphics[scale=0.50]{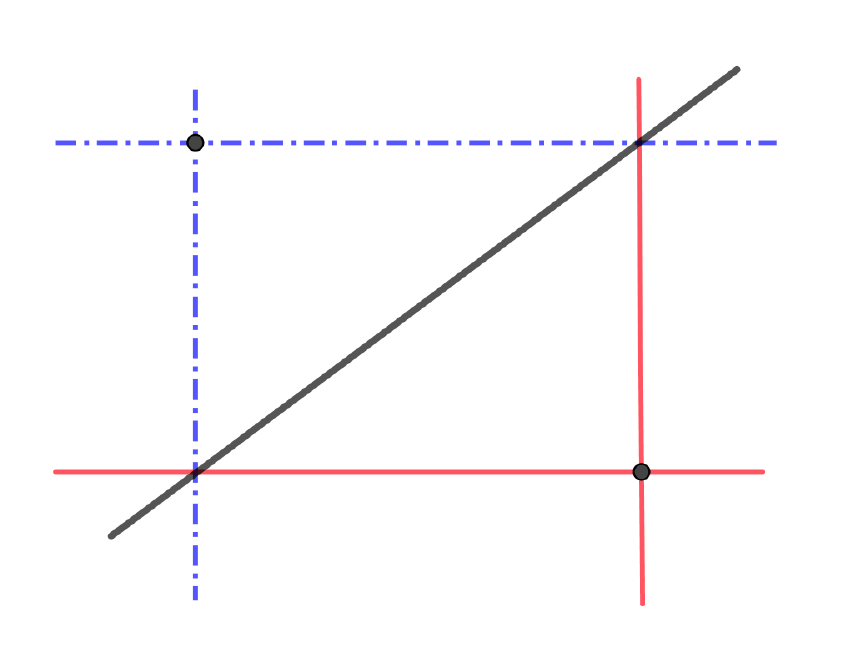}

\begin{picture}(0,0)
\put(-150,0){$\mathbb{P}^2_{\mathbb{C}}$}
\put(-20,0){$\tilde{Q}_{\mathbb{C}}$}
\put(130,0){$Q_{\mathbb{C}}$}

\put(95,90){$p_-$}
%\put(180,30){$p_+$}
\put(140,50){$H$}
\put(130,90){$\ell_-$}
\put(95,60){$\overline{\ell}_-$}

\put(-15,65){$\tilde{H}$}
\put(-10,100){$\ell_-$}
\put(-65,50){$\overline{\ell}_-$}
\put(-65,95){$E_-$}

\put(-160,85){$\Delta$}
\put(-160,55){$E_-$}
\put(-100,40){$\stackrel{\alpha'}{\longleftarrow}$}
\put(60,40){$\stackrel{\alpha}{\longrightarrow}$} 
\end{picture}

\caption{ A real birational map between $(\mathbb{P}^2_{\mathbb{C}},\sigma_{\mathbb{P}^2_{\mathbb{R}}})$ and $(Q_{\mathbb{C}},\sigma_Q)$.}
\label{QuadricP2}
\end{figure}
We obtain a diagram \[\xymatrix{(S_{3},\sigma_{3})\cong (\mathbb{P}_{\mathbb{C}}^{2}\setminus\Delta,\sigma_{\mathbb{P}_{\mathbb{R}}^{2}})& \ar[l]_-{\alpha'} (\tilde{Q}_{\mathbb{C}}\setminus(\tilde{H}\cup F_-),\sigma_{\tilde{Q}}) \ar[r]^-{\alpha} &  (Q_{\mathbb{C}}\setminus H,\sigma_{Q})\cong (S_2,\sigma_2) }\]in
which the left hand side induced morphism $\alpha'$ is an $\mathbb{S}^{1}$-equivariant
real isomorphism. The right hand side morphism $\alpha$ realizes
$(\tilde{Q}_{\mathbb{C}}\setminus(\tilde{H}\cup F_{-}),\sigma_{\tilde{Q}})$
as the $\mathbb{S}^{1}$-equivariant real affine modification of $(S_{2},\sigma_{2})$
obtained by blowing-up $p_{-}$ and removing the proper transforms
of the closures $\overline{O}^{\pm}$ of the two principal $\mathbb{G}_{m,\mathbb{C}}$-orbits
in $\pi_{2}^{-1}(-1)$ exchanged by $\sigma_{2}$ and intersecting
at $p_{-}$ (see Figure \ref{SphereMod}). \\

A real DPD-presentation of $(S_{3},\sigma_{3})$ can be determined
as follows. The smooth real quadric $(Q_{\mathbb{C}},\sigma_{Q})$
is isomorphic to the Galois double cover of $(\mathbb{P}_{\mathbb{C}}^{2},\sigma_{\mathbb{P}_{\mathbb{R}}^{2}})$
branched along the real conic $\Delta$. The commutative diagram \[\xymatrix{ S_2=Q_{\mathbb{C}}\setminus\{w=0\} \ar[r] \ar[d] & Q_{\mathbb{C}} \ar[d]^{[u:v:z:w]\mapsto [u:v:z]} \\ S_3=\mathbb{P}_{\mathbb{C}}^{2}\setminus\Delta \ar[r] & \mathbb{P}^2_{\mathbb{C}} }\]
then identifies $S_{3}$ with the quotient of $S_{2}\cong\{u^{2}+v^{2}+z^{2}=1\}\subset\mathbb{A}_{\mathbb{C}}^{3}$
by the antipodal involution $(u,v,z)\mapsto(-u,-v,-z)$. This involution
commutes with the real structure $\sigma_{2}$ on $S_{2}$ and the
quotient morphism $(S_{2},\sigma_{2})\rightarrow(S_{3},\sigma_{3})\cong(S_{2},\sigma_{2})/\mathbb{Z}_{2}$
is a real morphism, which is equivariant for the $\mathbb{S}^{1}$-actions
on $(S_{2},\sigma_{2})$ and $(S_{3},\sigma_{3})$. The real quotient
morphism $\pi_{2}:(S_{2},\sigma_{2})\rightarrow(C_{2},\tau_{2})$
thus descends to a real morphism 
\[
\pi_{3}:(S_{3},\sigma_{3})\cong(S_{2},\sigma_{2})/\mathbb{Z}_{2}\rightarrow(C_{3},\tau_{3})=(C_{2},\tau_{2})/\mathbb{Z}_{2}\cong(\mathbb{A}_{\mathbb{C}}^{1}=\mathrm{Spec}(\mathbb{C}[Z],\sigma_{\mathbb{A}_{\mathbb{R}}^{1}}),
\]
where $Z=2z^{2}-1$, which is the real quotient morphism of the induced
$\mathbb{S}^{1}$-action on $(S_{3},\sigma_{3})$. With this choice
of coordinate, a direct calculation shows that a real DPD-pair on
$(C_{3},\tau_{3})$ corresponding to $(S_{3},\sigma_{3})$ is $(D_{3},h_{3})=(\frac{1}{2}\{-1\},1-Z^{2})$. 

The image by $\pi_{3}$ of the real locus $\mathbb{RP}^{2}$ of $(S_{3},\sigma_{3})$
is the segment $[-1,1]$ of the real locus $\mathbb{R}$ of $(C_{3},\tau_{3})$.
The restriction of $\pi_{3}$ over the principal real affine open
subset $(\mathbb{A}_{\mathbb{C}}^{1}\setminus\{\pm1\},\sigma_{\mathbb{A}_{\mathbb{R}}^{1}}|_{\mathbb{A}_{\mathbb{C}}^{1}\setminus\{\pm1\}})$
of $(C_{3},\tau_{3})$ is a nontrivial $\mathbb{S}^{1}$-torsor. The
fibers of $\pi_{3}$ over the real points $-1$ and $1$ of $(C_{3},\tau_{3})$
are respectively of type b) and c) in Theorem \ref{prop:Real-Except-Orb-1}.
Their real loci consist respectively of a copy of $S^{1}$ on which
$S^{1}$ acts with stabilizer $\mu_{2}$ and a unique point $p$,
which is a fixed point of the induced $S^{1}$-action on the real
locus of $(S_{3},\sigma_{3})$. 

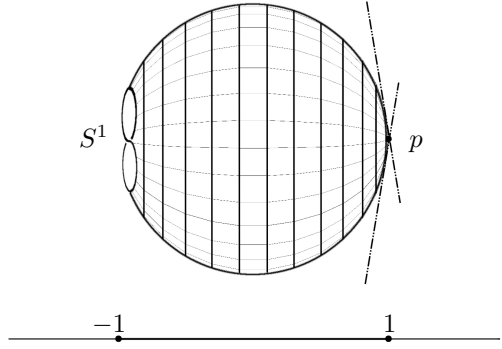
\begin{figure}[h!]

\includegraphics[scale=0.20]{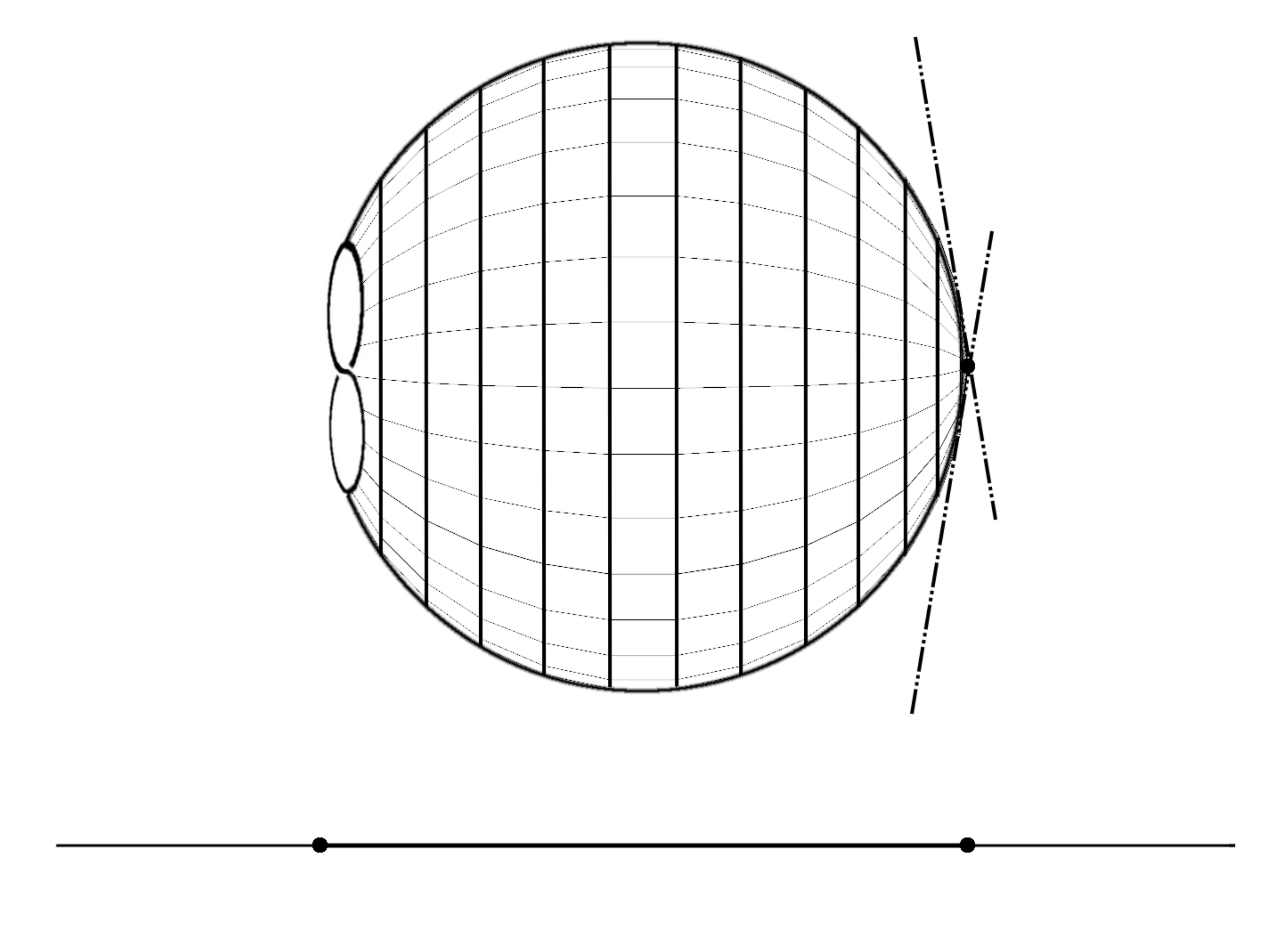}
\begin{picture}(0,0)
\put(-45,85){$p$}
\put(-170,85){$S^1$}
\put(-55,15){$1$}
\put(-165,15){$-1$}
\end{picture}

\caption{Projection of the real locus $\mathbb{RP}^2$ of $(S_3,\sigma_3)=(\mathbb{S}_{\mathbb{C}}^{2},\sigma_{\mathbb{S}^{2}})/\mathbb{Z}_{2}$ onto the interval $[-1,1]$. The dashed lines represent the closures $\overline{O}^{\pm}$ of the two principal $\mathbb{G}_{m,\mathbb{C}}$-orbits in $\pi_3^{-1}(1)$ exchanged by $\sigma_3$ and intersecting at the unique real $\mathbb{S}^1$-fixed point $p\subset\pi_3^{-1}(1)$.}
\label{P2Mod}
\end{figure}

\subsubsection{\label{subsec:KleinB-model}Equivariant rational models of the Klein
bottle }

The Klein bottle $K$ with its $S^{1}$-action is the $S^{1}$-equivariant
connected sum $\mathbb{RP}^{2}\sharp_{S^{1}}\mathbb{RP}^{2}$ of two
copies of $\mathbb{RP}^{2}$ endowed with the $S^{1}$-action defined
in the previous subsection. Namely, $\mathbb{RP}^{2}\sharp_{S^{1}}\mathbb{RP}^{2}$
is obtained by removing on each copy of $\mathbb{RP}^{2}$ an $S^{1}$-invariant
open disc containing the unique $S^{1}$-fixed point and gluing together
the resulting boundary circles in an $S^{1}$-equivariant way. Equivalenty,
$\mathbb{RP}^{2}\sharp_{S^{1}}\mathbb{RP}^{2}$ is obtained by the
$S^{1}$-equivariant gluing of two closed Moebius bands 
\[
\overline{B}=(S^{1}\times[-1,1])/(z,u)\sim(-z,-u)
\]
with the $S^{1}$-action as in the previous subsection along their
boundary circles. The resulting $S^{1}$-action on $K$ has two non
principal orbits isomorphic to $S^{1}$ with stabilizer $\mu_{2}$,
and with $S^{1}$-invariant tubular neighborhoods diffeomorphic to
open Moebius bands. The orbit space is again closed interval, each
of its ends corresponding to a non-principal orbit. 

\begin{figure}[h!]

\includegraphics[scale=0.60]{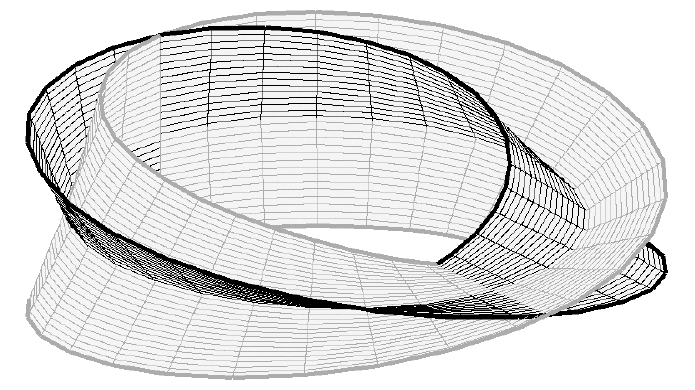} \hspace{1cm}
\includegraphics[scale=0.60]{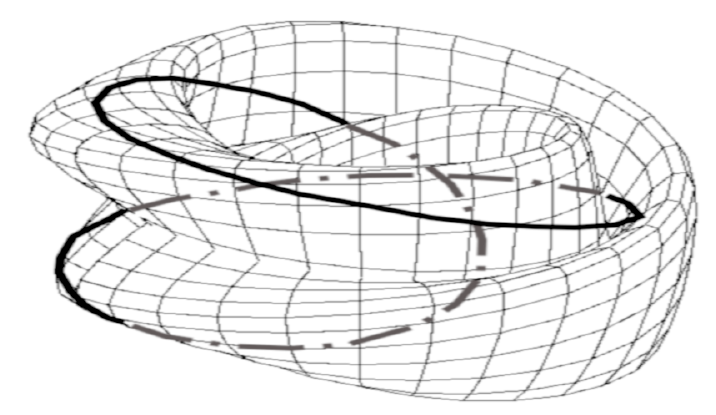}

\caption{Two Moebius bands glued along their boundaries.}
\label{KleinGlue}
\end{figure}

A projective rational model of $K$ is obtained as follows: let $(\mathbb{P}_{\mathbb{C}}^{2}=\mathrm{Proj}(\mathbb{C}[u,v,z]),\sigma_{\mathbb{P}_{\mathbb{R}}^{2}})$
be endowed with the $\mathbb{S}^{1}$-action defined by the projective
representation induced by the direct sum of the representation $\mathbb{S}^{1}\rightarrow SO_{2}(\mathbb{R})$
with the trivial $1$-dimensional representation. The blow-up of $(\mathbb{P}_{\mathbb{C}}^{2},\sigma_{\mathbb{P}_{\mathbb{R}}^{2}})$
at the real $\mathbb{S}^{1}$-fixed point $[0:0:1]$ is the real Hirzebruch
surface $(\mathbb{F}_{1,\mathbb{C}}=\mathbb{P}(\mathcal{O}_{\mathbb{P}_{\mathbb{R}}^{1}}\oplus\mathcal{O}_{\mathbb{P}_{\mathbb{R}}^{1}}(-1)),\sigma_{\mathbb{F}_{1}})$
in which the exceptional divisor $E\cong(\mathbb{P}_{\mathbb{C}}^{1},\sigma_{\mathbb{P}_{\mathbb{R}}^{1}})$
is the section with self-intersection $-1$ of the $\mathbb{P}^{1}$-bundle
structure. The $\mathbb{S}^{1}$-action on $(\mathbb{P}_{\mathbb{C}}^{2},\sigma_{\mathbb{P}_{\mathbb{R}}^{2}})$
lifts to an action on $(\mathbb{F}_{1,\mathbb{C}},\sigma_{\mathbb{F}_{1}})$
and the real locus of $(\mathbb{F}_{1,\mathbb{C}},\sigma_{\mathbb{F}_{1}})$
endowed with the induced $S^{1}$-action is diffeomorphic to $K\simeq\mathbb{RP}^{2}\sharp_{S^{1}}\mathbb{RP}^{2}$
endowed with the $S^{1}$-action defined above. 

An affine model of $K$ is in turn obtained from $(\mathbb{F}_{1,\mathbb{C}},\sigma_{\mathbb{F}_{1}})$
by removing the union of the proper transform of the conic $\Delta=\left\{ u^{2}+v^{2}+z^{2}=0\right\} \subset\mathbb{P}_{\mathbb{C}}^{2}$
and of the proper transforms of the pair of non-real lines 
\[
\ell=\{u+iv=0\}\quad\textrm{and}\quad\sigma_{\mathbb{P}_{\mathbb{R}}^{2}}(\ell)=\{u-iv=0\}
\]
of $(\mathbb{P}_{\mathbb{C}}^{2},\sigma_{\mathbb{P}_{\mathbb{R}}^{2}})$
passing through $[0:0:1]$. Indeed, $\Delta\cup\ell\cup\sigma_{\mathbb{P}_{\mathbb{R}}^{2}}(\ell)$
is a real $\mathbb{S}^{1}$-invariant closed subset of $(\mathbb{P}_{\mathbb{C}}^{2},\sigma_{\mathbb{P}_{\mathbb{R}}^{2}})$
with $[0:0:1]$ as a unique real point, whose proper transform $B$
in $(\mathbb{F}_{1,\mathbb{C}},\sigma_{\mathbb{F}_{1}})$ is an ample
$\mathbb{S}^{1}$-invariant curve with empty real locus. So $(S_{4},\sigma_{4})=(\mathbb{F}_{1,\mathbb{C}}\setminus B,\sigma_{\mathbb{F}_{1}}|_{\mathbb{F}_{1,\mathbb{C}}\setminus B})$
endowed with the induced $\mathbb{S}^{1}$-action is a real affine
surface whose real locus is diffeomorphic to $K$. By construction,
the inclusion $(S_{4},\sigma_{4})\hookrightarrow(\mathbb{F}_{1,\mathbb{C}},\sigma_{\mathbb{F}_{1}})$
is an $\mathbb{S}^{1}$-equivariant birational diffeomorphism. 

A alternative projective model of $K$ is obtained from the quadric
$(Q_{\mathbb{C}},\sigma_{Q})$ of subsection \ref{subsec:ModelSphere}
by blowing-up the two real $\mathbb{S}^{1}$-fixed points $p_{\pm}=\left[0:0:\pm1:1\right]$
with respective exceptional divisors $E_{\pm}\cong(\mathbb{P}_{\mathbb{C}}^{1},\sigma_{\mathbb{P}_{\mathbb{R}}^{1}})$.
Letting $\beta:(\hat{Q}_{\mathbb{C}},\sigma_{\hat{Q}})\rightarrow(Q_{\mathbb{C}},\sigma_{Q})$
be the real blow-up morphism, the $\mathbb{S}^{1}$-action on $(Q_{\mathbb{C}},\sigma_{Q})$
lifts to an action on $(\hat{Q}_{\mathbb{C}},\sigma_{\hat{Q}})$ for
which the real locus of $(\hat{Q}_{\mathbb{C}},\sigma_{\hat{Q}})$
endowed with the induced $S^{1}$-action is equivariantly diffeomorphic
to the connected sum $\mathbb{RP}^{2}\sharp_{S^{1}}S^{2}\sharp_{S^{1}}\mathbb{RP}^{2}\simeq\mathbb{RP}^{2}\sharp_{S^{1}}\mathbb{RP}^{2}$,
hence to the Klein bottle $K$ endowed with the $S^{1}$-action defined
above. As in the previous subsection, the proper transforms in $(\hat{Q}_{\mathbb{C}},\sigma_{\hat{Q}})$
of the curves $\ell_{-}=\{u+iv=z+w=0\}$ and $\overline{\ell}_{-}=\{u-iv=z+w=0\}$
in $(Q_{\mathbb{C}},\sigma_{Q})$ are a pair of non-real disjoint
smooth rational curves with self-intersection $-1$, exchanged by
the real structure $\sigma_{\hat{Q}}$. Their union is an $\mathbb{S}^{1}$-invariant
real closed subset $F_{-}$ of $(\hat{Q}_{\mathbb{C}},\sigma_{\hat{Q}})$
and the contraction of $F_{-}$ is an $\mathbb{S}^{1}$-equivariant
birational diffeomorphism $\beta':(\hat{Q}_{\mathbb{C}},\sigma_{\hat{Q}})\rightarrow(\mathbb{F}_{1,\mathbb{C}},\sigma_{\mathbb{F}_{1}})$
which maps the proper transform $\hat{H}\subset\hat{Q}_{\mathbb{C}}$
of the curve $H=\left\{ w=0\right\} \subset Q_{\mathbb{C}}$ onto
the proper transform in $(\mathbb{F}_{1,\mathbb{C}},\sigma_{\mathbb{F}_{1}})$
of the conic $\Delta=\{u^{2}+v^{2}+z^{2}=0\}\subset\mathbb{P}_{\mathbb{C}}^{2}$.
The proper transforms in $(\hat{Q}_{\mathbb{C}},\sigma_{\hat{Q}})$
of the curves $\ell_{+}=\{u+iv=z-w=0\}$ and $\overline{\ell}_{+}=\{u-iv=z-w=0\}$
in $(Q_{\mathbb{C}},\sigma_{Q})$ are also a pair of non-real disjoint
smooth rational curves with self-intersection $-1$, exchanged by
the real structure $\sigma_{\hat{Q}}$. Their union is an $\mathbb{S}^{1}$-invariant
real closed subset $F_{+}$ of $(\hat{Q}_{\mathbb{C}},\sigma_{\hat{Q}})$
whose image by $\beta':(\hat{Q}_{\mathbb{C}},\sigma_{\hat{Q}})\rightarrow(\mathbb{F}_{1,\mathbb{C}},\sigma_{\mathbb{F}_{1}})$
is equal to the union of the proper transforms in $(\mathbb{F}_{1,\mathbb{C}},\sigma_{\mathbb{F}_{1}})$
of the lines $\ell$ and $\sigma_{\mathbb{P}_{\mathbb{R}}^{2}}(\ell)$
of $(\mathbb{P}_{\mathbb{C}}^{2},\sigma_{\mathbb{P}_{\mathbb{R}}^{2}})$
(see Figure \ref{QuadricF1}). 

\begin{figure}[h!]
\includegraphics[scale=0.45]{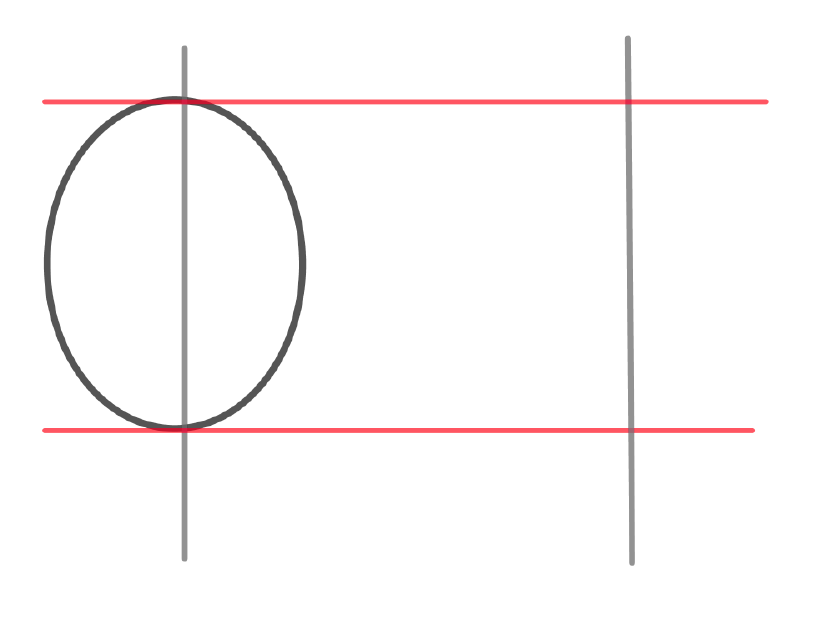} \hspace{0.5cm} 
\includegraphics[scale=0.35]{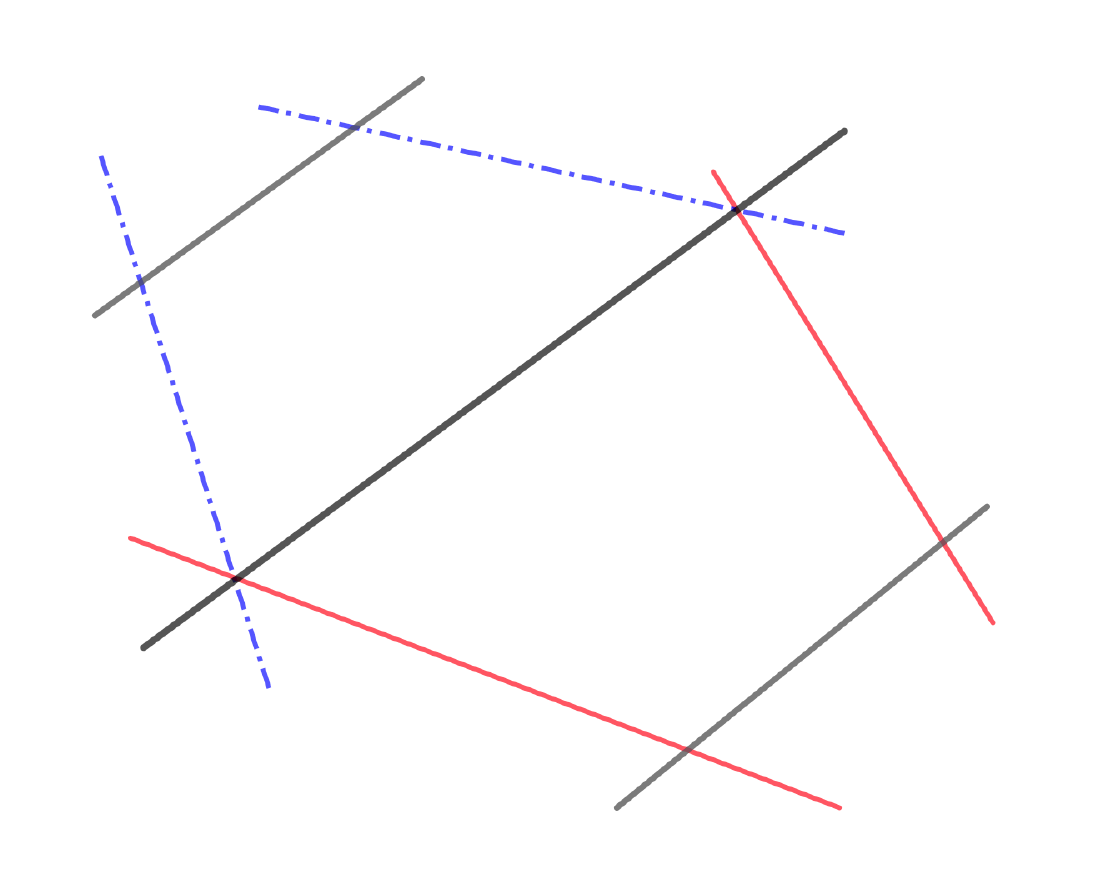} \hspace{0.5cm}
\includegraphics[scale=0.45]{./Quadric.pdf}

\begin{picture}(0,0)
\put(-150,10){$\mathbb{F}_{1,\mathbb{C}}$}
\put(-20,10){$\hat{Q}_{\mathbb{C}}$}
\put(120,10){$Q_{\mathbb{C}}$}

\put(85,85){$p_-$}
\put(160,25){$p_+$}
\put(130,50){$H$}
\put(120,80){$\ell_-$}
\put(85,55){$\overline{\ell}_-$}
\put(120,25){$\ell_+$}
\put(160,55){$\overline{\ell}_+$}

\put(-15,65){$\hat{H}$}
\put(-10,90){$\ell_-$}
\put(-55,55){$\overline{\ell}_-$}
\put(-45,83){$E_-$}
\put(30,60){$\overline{\ell}_+$}
\put(-20,25){$\ell_+$}
\put(25,25){$E_+$}

\put(-190,55){$\Delta$}
\put(-160,55){$E_-$}
\put(-100,55){$E_+$}
\put(-150,28){$\ell_+=\ell$}
\put(-158,88){$\overline{\ell}_+=\sigma_{\mathbb{P}^2_{\mathbb{R}}}(\ell)$}
\put(-80,50){$\stackrel{\beta'}{\longleftarrow}$}
\put(55,50){$\stackrel{\beta}{\longrightarrow}$} 
\end{picture}

\caption{ A real birational map between $(\mathbb{F}_{1,\mathbb{C}},\sigma_{\mathbb{F}_1})$ and $(Q_{\mathbb{C}},\sigma_Q)$.}
\label{QuadricF1}
\end{figure}

We obtain a diagram \[\xymatrix{ (S_4,\sigma_4)=(\mathbb{F}_{1,\mathbb{C}}\setminus B,\sigma_{\mathbb{F}_{1}}) & \ar[l]_{\beta '}(\hat{Q}_{\mathbb{C}}\setminus(\hat{H}\cup F_-\cup F_+),\sigma_{\hat{Q}}) \ar[r]^{\beta} &   (Q_{\mathbb{C}}\setminus H,\sigma_{Q})\cong (S_2,\sigma_2)}\] in
which the left hand side induced morphism $\beta'$ is an $\mathbb{S}^{1}$-equivariant
real isomorphism. The right hand side morphism $\beta$ realizes the
real affine surface $(\hat{Q}_{\mathbb{C}}\setminus(\hat{H}\cup F_{-}\cup F_{+}),\sigma_{\hat{Q}})$
as the $\mathbb{S}^{1}$-equivariant real affine modification of $(S_{2},\sigma_{2})$
obtained by blowing-up $p_{-}$ and $p_{+}$ and removing the proper
transforms of the closures $\overline{O}_{\pm1}^{\pm}$ of the two
principal $\mathbb{G}_{m,\mathbb{C}}$-orbits in $\pi_{2}^{-1}(\pm1)$
exchanged by $\sigma_{2}$ and intersecting at $p_{\pm}$ (see Figure
\ref{SphereMod}). \\

The $\mathbb{S}^{1}$-equivariant affine modification $\beta:(S_{4},\sigma_{4})\rightarrow(S_{2},\sigma_{2})$
can be made explicit as follows. Let $(Q_{2},\sigma_{2}')$ be the
smooth surface in $\mathrm{Spec}(\mathbb{C}[x,y,z])$ with equation
$xy=1-z^{2}$, endowed with the real structure defined as the composition
of the involution $(x,y,z)\mapsto(y,x,z)$ with the complex conjugation.
The isomorphism
\[
Q_{2}\rightarrow S_{2}\subset\mathrm{Spec}(\mathbb{C}[u,v,z]),\quad\left(x,y,z\right)\mapsto(\frac{x+y}{2},\frac{x-y}{2i},z)
\]
induces an real isomorphism $(Q_{2},\sigma'_{2})\cong(S_{2},\sigma_{2})$
which is equivariant for the $\mathbb{S}^{1}$-action on $(Q_{2},\sigma'_{2})$
given by the hyperbolic $\mathbb{G}_{m,\mathbb{C}}$-action $\mu(t,(x,y,z))=(tx,t^{-1}y,z)$.
Via this isomorphism, the $\mathbb{S}^{1}$-equivariant real affine
modification of $(S_{2},\sigma_{2})$ described geometrically above
coincides with the affine modification of $(Q_{2},\sigma'_{2})$ with
center at the real $\mathbb{S}^{1}$-invariant closed subscheme with
defining ideal $I=(x,y)^{2}$ and with real $\mathbb{S}^{1}$-invariant
principal divisor $\mathrm{div}(xy)$. It follows that $(S_{4},\sigma_{4})$
is $\mathbb{S}^{1}$-equivariantly isomorphic to the complex affine
surface in $\mathrm{Spec}(\mathbb{C}[x^{\pm1},y,z])$ defined by the
equation $xy^{2}=1-z^{2}$ , endowed with the real structure given
by the composition of the involution $(x,y,z)\mapsto(x^{-1},xy,z)$
with the complex conjugation, equipped the $\mathbb{S}^{1}$-action
given by the hyperbolic $\mathbb{G}_{m,\mathbb{C}}$-action $\mu(t,(x,y,z))=(t^{2}x,t^{-1}y,z)$.

We deduce from this description that the real quotient morphism of
$(S_{4},\sigma_{4})$ is the projection 
\[
\pi_{4}=\mathrm{pr}_{z}:(S_{4},\sigma_{4})\rightarrow(C_{4},\tau_{4})=(\mathbb{A}_{\mathbb{C}}^{1}=\mathrm{Spec}(\mathbb{C}[z],\sigma_{\mathbb{A}_{\mathbb{R}}^{1}})
\]
and that a real DPD-pair on $(C_{4},\tau_{4})$ corresponding to $(S_{4},\sigma_{4})$
is $(D_{4},h_{4})=(\frac{1}{2}\{-1\}+\frac{1}{2}\{1\},1-z^{2})$.
The image by $\pi_{4}$ of the real locus of $(S_{4},\sigma_{4})$
is the segment $[-1,1]$ of the real locus $\mathbb{R}$ of $(C_{4},\tau_{4})$.
The restriction of $\pi_{4}$ over the principal real affine open
subset $(\mathbb{A}_{\mathbb{C}}^{1}\setminus\{\pm1\},\sigma_{\mathbb{A}_{\mathbb{R}}^{1}}|_{\mathbb{A}_{\mathbb{C}}^{1}\setminus\{\pm1\}})$
of $(C_{4},\tau_{4})$ is a nontrivial $\mathbb{S}^{1}$-torsor. The
fibers of $\pi_{4}$ over the real points $\pm1$ of $(C_{4},\tau_{4})$
are of type b) in Theorem \ref{prop:Real-Except-Orb-1}. Their real
loci consist of a copy of $S^{1}$ on which $S^{1}$ acts with stabilizer
$\mu_{2}$. 

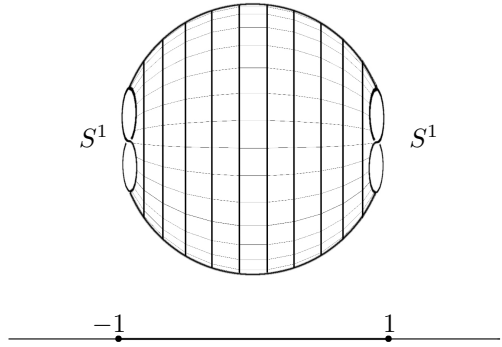
\begin{figure}[h!]

\includegraphics[scale=0.20]{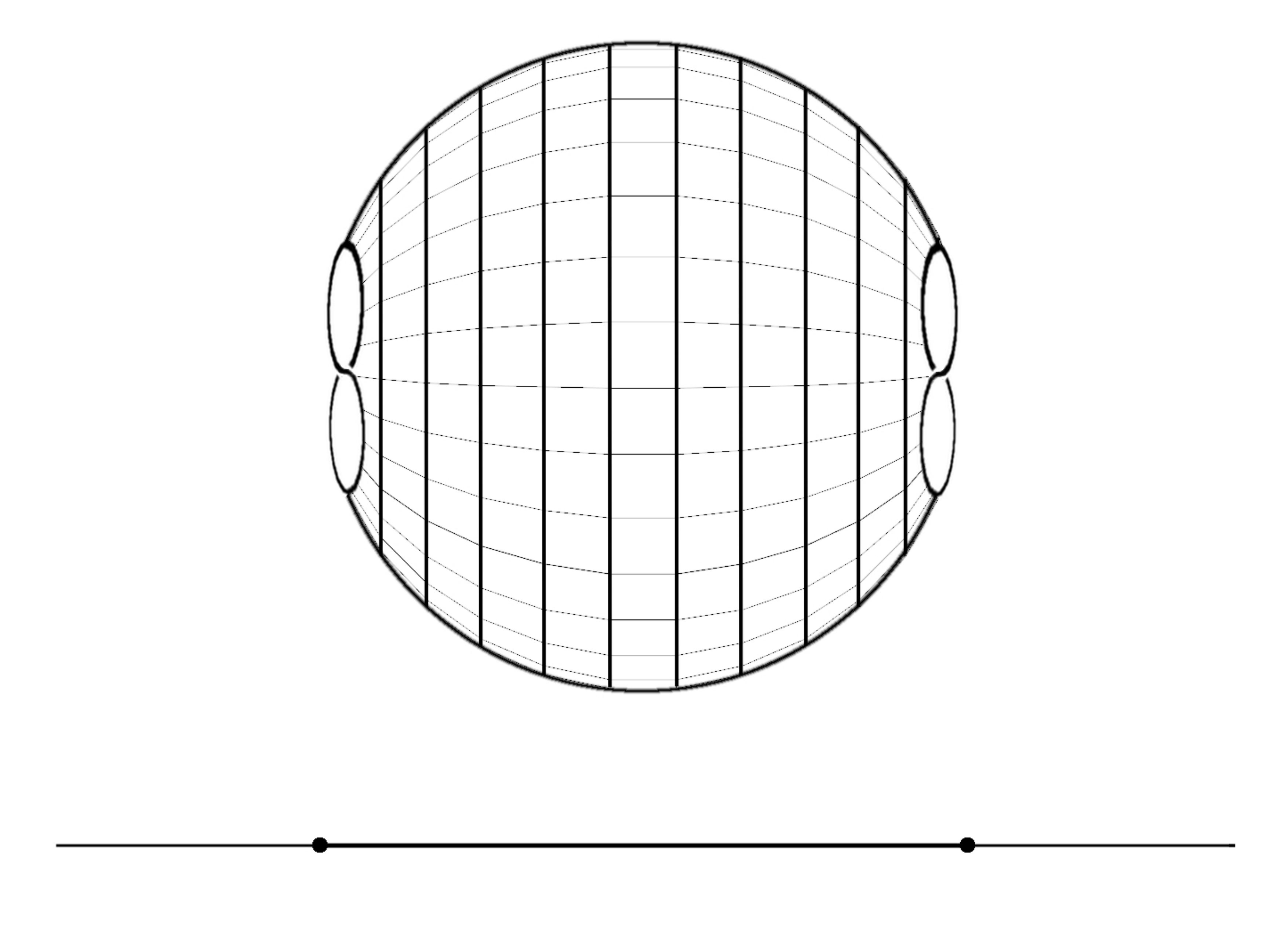}
\begin{picture}(0,0)
\put(-45,85){$S^1$}
\put(-170,85){$S^1$}
\put(-55,15){$1$}
\put(-165,15){$-1$}
\end{picture}

\caption{Projection of the real locus $K$ of $(S_4,\sigma_4)$ onto the interval $[-1,1]$.}
\label{KleinMod}
\end{figure}

\subsection{Uniqueness of models up to equivariant birational diffeomorphism }

This subsection is devoted to the proof of the following result which
implies Theorem \ref{thm:MainThm-Model}: 
\begin{prop}
Every smooth rational quasi-projective real surface with an effective
$\mathbb{S}^{1}$-action and whose real locus is a compact connected
manifold of dimension $2$ without boundary is $\mathbb{S}^{1}$-equivariantly
birationally diffeomorphic to one of the affine models constructed
in subsection \ref{subsec:Rational-affine-models-Construct}, summarized
in Table \ref{RatMods} below. 
\end{prop}

\begin{figure}[htb!]
\begin{center}
\begin{tabular}{|c|c|c|c|c|}
\hline 
Real locus & $S^{1}\times S^{1}$ & $S^{2}$ & $\mathbb{RP}^{2}$ & $K$\tabularnewline
\hline 
Rational Mmodel & $S_{1}=Q_{1,\mathbb{C}}\times\mathbb{S}_{\mathbb{C}}^{1}$ & $S_{2}=\mathbb{S}_{\mathbb{C}}^{2}$ & $S_{3}=\mathbb{S}_{\mathbb{C}}^{2}/\mathbb{Z}_{2}$ & $S_{4}=\{xy^{2}=1-z^{2}\}$\tabularnewline
\hline 
Real categorical quotient  & $(Q_{1,\mathbb{C}},\sigma_{Q_{1}})$ & $(\mathbb{A}_{\mathbb{C}}^{1},\sigma_{\mathbb{A}_{\mathbb{R}}^{1}})$ & $(\mathbb{A}_{\mathbb{C}}^{1},\sigma_{\mathbb{A}_{\mathbb{R}}^{1}})$ & $(\mathbb{A}_{\mathbb{C}}^{1},\sigma_{\mathbb{A}_{\mathbb{R}}^{1}})$\tabularnewline
\hline 
Image of real locus  & $S^{1}$ & $[-1,1]$ & $[-1,1]$ & $[-1,1]$\tabularnewline
\hline 
Real DPD-pair $(D,h)$ & $(0,1)$ & $(0,1-z^{2})$ & $(\frac{1}{2}\{-1\},1-z^{2})$ & ${\displaystyle (\frac{1}{2}\{-1\}+\frac{1}{2}\{1\},1-z^{2})}$\tabularnewline
\hline 
\end{tabular}
\par\end{center}

\caption{Rational affine models of compact surfaces with $S^1$-actions. The notation  $(Q_{1,\mathbb{C}},\sigma_{Q_1})$ refers to the underlying real algebraic variety of $\mathbb{S}^1$, that is, the complexification of the smooth affine curve in $\mathbb{A}^2_{\mathbb{R}}=\mathrm{Spec}(\mathbb{R}[u,v])$ with equation $u^2+v^2=1$.}
\label{RatMods}
\end{figure}

The scheme of the proof is the following. In Lemma \ref{lem:Quas-proj-2-affine}
below, we first establish that every smooth rational quasi-projective
model with $\mathbb{S}^{1}$-action of the torus $T$, or the sphere
$S^{2}$, or the plane $\mathbb{RP}^{2}$ or the Klein bottle $K$
is $\mathbb{S}^{1}$-equivariantly birationally diffemorphic to an
affine one. Then in Lemma \ref{lem:CompactLocus-case-disjunction},
we split the study of the affine case into two subcases according
to the nature of the image of the real locus by the real quotient
morphism. These subcases are finally studied separately in subsections
\ref{subsec:First-case} and \ref{subsec:Second-case}. 
\begin{lem}
\label{lem:Quas-proj-2-affine}Let $(X,\Sigma)$ be a smooth rational
real quasi-projective surface with an effective $\mathbb{S}^{1}$-action
and whose real locus is a compact connected manifold of dimension
$2$ without boundary. Then $(X,\Sigma)$ is $\mathbb{S}^{1}$-equivariantly
birationally diffeomorphic to a smooth rational real affine surface
$(S,\sigma)$ with $\mathbb{S}^{1}$-action.
\end{lem}

\begin{proof}
By Sumuhiro equivariant completion theorem \cite{Su75} and equivariant
desingularization results for normal surfaces with $\mathbb{G}_{m,\mathbb{C}}$-actions
\cite{OrWa71}, there exists a smooth real projective surface $(\overline{X},\overline{\Sigma})$
with $\mathbb{S}^{1}$-action and an $\mathbb{S}^{1}$-equivariant
open embedding $(X,\Sigma)\hookrightarrow(\overline{X},\overline{\Sigma})$.
Since $(X,\Sigma)$ is rational, so is $(\overline{X},\overline{\Sigma})$.
By a result of Comessatti \cite[p. 257]{Co14}, the real locus of
$(\overline{X},\overline{\Sigma})$ is a connected compact smooth
surface without boundary, either non-orientable, or orientable and
diffeomorphic to $T$ or $S^{2}$. Since the real locus of $(X,\Sigma)$
is itself connected and compact, it follows that the real loci of
$(X,\Sigma)$ and $(\overline{X},\overline{\Sigma})$ coincide, so
that $(X,\Sigma)\hookrightarrow(\overline{X},\overline{\Sigma})$
is a birational diffeomorphism. By \cite[Theorem 1.6]{Su75}, there
exists a very ample $\mathbb{S}^{1}$-linearized invertible sheaf
$\mathcal{L}$ on $(\overline{X},\overline{\Sigma})$. This yields
in particular a representation of $\mathbb{S}^{1}$ into the group
of linear automorphism of $H^{0}(\overline{X},\mathcal{L})$. The
underlying representation of $\mathbb{G}_{m,\mathbb{C}}$ splits as
a direct sum of $n_{1}\geq1$ non trivial diagonal representations
of the form 
\[
t\cdot(x_{i},y_{i})=(t^{m_{i}}x_{i},t^{-m_{i}}y_{i}),\;m_{i}\in\mathbb{Z}_{>0},\quad i=0,\ldots,n_{1}-1
\]
on which the real structure is given by the composition of the involution
$(x_{i},y_{i})\mapsto(y_{i},x_{i})$ with the complex conjugation,
and $n_{2}\geq0$ trivial $1$-dimensional representations. The $\mathbb{G}_{m,\mathbb{C}}$-equivariant
closed embedding 
\[
\overline{X}\hookrightarrow\mathbb{P}(H^{0}(\overline{X},\mathcal{L}))\cong\mathrm{Proj}(\mathbb{C}[x_{0},y_{0},\ldots x_{n_{1}-1},y_{n_{1}-1},z_{1},\ldots,z_{n_{2}}])
\]
then becomes a real $\mathbb{S}^{1}$-equivariant closed embedding
for the real structure $\overline{\Sigma}$ on $\overline{X}$ and
the real structure on $\mathbb{P}(H^{0}(\overline{X},\mathcal{L}))$
defined as the composition of the involution 
\[
(x_{0},y_{0},\ldots x_{n_{1}-1},y_{n_{1}-1},z_{1},\ldots,z_{n_{2}})\mapsto(y_{0},x_{0},\ldots y_{n_{1}-1},x_{n_{1}-1},z_{1},\ldots,z_{n_{2}})
\]
with the complex conjugation. The quadric $Q\subset\mathbb{P}(H^{0}(\overline{X},\mathcal{L}))$
with equation $\sum_{i=0}^{n_{1}-1}x_{i}y_{i}+\sum_{j=1}^{n_{2}}z_{j}^{2}=0$
is a real ample $\mathbb{S}^{1}$-invariant divisor on $\mathbb{P}(H^{0}(\overline{X},\mathcal{L}))$
with empty real locus. Since the real locus of $(\overline{X},\overline{\Sigma})$
is not empty, $\overline{X}$ is not contained in $Q$. It follows
that $(S,\sigma)=(\overline{X}\setminus Q,\overline{\Sigma}|_{\overline{X}\setminus Q})$
is a smooth rational real affine surface with $\mathbb{S}^{1}$-action.
By construction, the open inclusion $(S,\sigma)\hookrightarrow(\overline{X},\overline{\Sigma})$
is an $\mathbb{S}^{1}$-equivariant birational diffeomorphism. 
\end{proof}
The following lemma divides in turn the study of the affine case into
two sub-cases:
\begin{lem}
\label{lem:CompactLocus-case-disjunction} Let $(S,\sigma)$ be a
smooth rational real affine surface with an effective $\mathbb{S}^{1}$-action
and whose real locus is a connected compact surface without boundary.
Let $\pi:(S,\sigma)\rightarrow(C,\tau)$ be the real quotient morphism
for the $\mathbb{S}^{1}$-action. Then the following alternative holds: 

a) $(C,\tau)$ is a real affine open subset of $(Q_{1,\mathbb{C}},\sigma_{Q_{1}})$
and its real locus is equal to that of $(Q_{1,\mathbb{C}},\sigma_{Q_{1}})$. 

b) $(C,\tau)$ is a real affine open subset of the real affine line
$(\mathbb{A}_{\mathbb{C}}^{1},\sigma_{\mathbb{A}_{\mathbb{R}}^{1}})$
and its real locus is a closed interval. 
\end{lem}

\begin{proof}
The curve $C$ is rational because $S$ is rational. Since the real
locus of $(S,\sigma)$ is nonempty, connected and compact, its image
by $\pi:(S,\sigma)\rightarrow(C,\tau)$ is a nonempty connected compact
subset of the real locus of $(C,\tau)$. The smooth real projective
model of $(C,\tau)$ is thus isomorphic to the real projective line
$(\mathbb{P}_{\mathbb{C}}^{1},\sigma_{\mathbb{P}_{\mathbb{R}}^{1}})$.
If $\mathbb{P}_{\mathbb{C}}^{1}\setminus C$ contains a real point
of $(\mathbb{P}_{\mathbb{C}}^{1},\sigma_{\mathbb{P}_{\mathbb{R}}^{1}})$,
then $C$ is isomorphic to a real affine open subset of the real affine
line. Being connected and compact, its real locus is then a closed
interval. Otherwise, since the inclusion $(C,\tau)\hookrightarrow(\mathbb{P}_{\mathbb{C}}^{1},\sigma_{\mathbb{P}_{\mathbb{R}}^{1}})$
is a real morphism and $C$ is affine, $\mathbb{P}_{\mathbb{C}}^{1}\setminus C$
is not empty and consists of pairs of non-real points of $\mathbb{P}_{\mathbb{C}}^{1}$
which are exchanged by the real structure $\sigma_{\mathbb{P}_{\mathbb{R}}^{1}}$.
Since the complement of a pair of such points $q$ and $\sigma_{\mathbb{P}_{\mathbb{R}}^{1}}(q)$
is isomorphic to the real affine quadric $(Q_{1,\mathbb{C}},\sigma_{Q_{1}})$,
it follows that $(C,\tau)$ is isomorphic to a real affine open subset
of $(Q_{1,\mathbb{C}},\sigma_{Q_{1}})$, and since the real locus
of $\mathbb{P}_{\mathbb{C}}^{1}\setminus C$ is empty, it follows
that the real locus of $(C,\tau)$ is equal to that of $(Q_{1,\mathbb{C}},\sigma_{Q_{1}})$. 
\end{proof}

\subsubsection{\label{subsec:First-case}First case: $(C,\tau)$ is a real affine
open subset of $(Q_{1,\mathbb{C}},\sigma_{Q_{1}})$ with real locus
equal to $S^{1}$. }
\begin{prop}
\label{prop:Torus-diffBir} Let $(S,\sigma)$ be a smooth rational
real affine surface with an effective $\mathbb{S}^{1}$-action and
whose real locus is a connected compact surface without boundary.
Let $\pi:(S,\sigma)\rightarrow(C,\tau)$ be the real quotient morphism
and assume that $(C,\tau)$ is a real affine open subset of $(Q_{1,\mathbb{C}},\sigma_{Q_{1}})$
whose real locus is equal to that of $(Q_{1,\mathbb{C}},\sigma_{Q_{1}})$.
Then $(S,\sigma)$ is $\mathbb{S}^{1}$-equivariantly birationally
diffeomorphic to $(Q_{1,\mathbb{C}},\sigma_{Q_{1}})\times\mathbb{S}^{1}$
on which $\mathbb{S}^{1}$ acts by translations on the second factor. 
\end{prop}

\begin{proof}
Let $(D,h)$ be a regular real DPD-pair on $(C,\tau)$ corresponding
to $(S,\sigma)$. Since $h$ is a $\tau$-invariant rational function
on $C$, it is also a $\sigma_{Q_{1}}$-invariant rational function
on $Q_{1,\mathbb{C}}$. We claim that by changing $h$ for some rational
function of the form $f\tau^{*}fh$, where $f$ is a rational function
on $Q_{1,\mathbb{C}}$, and changing $D$ accordingly by $D+\mathrm{div}(f|_{C})$,
we can assume that $h$ is the restriction to $(C,\tau)$ of a real
regular function on $(Q_{1,\mathbb{C}},\sigma_{Q_{1}})$ whose zero
locus on $Q_{1,\mathbb{C}}$ consists of real points only. Indeed,
since $h$ is $\sigma_{Q_{1}}$-invariant, its poles on $Q_{1}$ are
either real points of $(Q_{1,\mathbb{C}},\sigma_{Q_{1}})$ or pairs
of non-real points exchanged by $\sigma_{Q_{1}}$. So up to changing
$h$ for $f\tau^{*}fh$ and $D$ for $D+\mathrm{div}(f|_{C})$ for
a suitable regular function $f$ on $Q_{1,\mathbb{C}}$, we can assume
from the very beginning that $h$ is the restriction of a real regular
function on $(Q_{1,\mathbb{C}},\sigma_{Q_{1}})$. Let $q=(z_{1},z_{2})$
and $\sigma_{Q_{1}}(q)=(\overline{z}_{1},\overline{z}_{2})=\overline{q}$
be a pair of non-real points of $Q_{1,\mathbb{C}}$ at which $h$
vanishes. The restrictions to $Q_{1,\mathbb{C}}$ of the regular functions
\[
F_{q}=(v-z_{2})-i(u-z_{1})\quad\textrm{and}\quad F_{\overline{q}}=(v-\overline{z}_{2})+i(u-\overline{z}_{1})
\]
on $\mathbb{A}_{\mathbb{C}}^{2}=\mathrm{Spec}(\mathbb{C}[u,v])$ are
regular functions $f_{q}$ and $f_{\overline{q}}$ on $Q_{1,\mathbb{C}}$
such that $\mathrm{div}(f_{q})=q$ and $\mathrm{div}(f_{\overline{q}})=\overline{q}$.
Furthermore, since $\sigma_{Q_{1}}^{*}f_{q}=f_{\overline{q}}$ it
follows that for $\delta=\mathrm{ord}_{q}(h)=\mathrm{ord}_{\overline{q}}(h)$,
$f_{q}^{-\delta}\sigma_{Q_{1}}^{*}f_{q}^{-\delta}h$ is a real regular
function on $(Q_{1,\mathbb{C}},\sigma_{Q_{1}})$ which does not vanish
at $q$ and $\overline{q}$. The pair 
\[
(D',h')=(D-\delta\mathrm{div}(f_{q}|_{C}),f_{q}^{-\delta}\sigma_{Q_{1}}^{*}f_{q}^{-\delta}h)
\]
is then a regular real DPD-pair on $(C,\tau)$ which defines a smooth
real affine surface $\mathbb{S}^{1}$-equivariantly isomorphic to
$(S,\sigma)$ by Theorem \ref{thm:MainThm} 1). The desired regular
real DPD-pair is then obtained by applying this construction to the
finitely many pairs of non-real points of $(Q_{1,\mathbb{C}},\sigma_{Q_{1}})$
exchanged by $\sigma_{Q_{1}}$ at which $h$ vanishes. 

The set of non-real points $q$ of $(C,\tau)$ such that either $D(q)\neq0$
or $D(\tau(q))\neq0$ is a finite real subset $Z$ of $(C,\tau)$.
Its complement $(U,\tau|_{U})$ is a real affine open subset of $(C,\tau)$
and the restriction $(D|_{U},h|_{U})$ of $(D,h)$ is a regular real
DPD-pair defining a smooth real affine surface $\mathbb{S}^{1}$-equivariantly
isomorphic to $(\pi^{-1}(U),\sigma|_{\pi^{-1}(U)})$. Since $Z$ consists
of non-real point of $(C,\tau)$ only, the inclusion of $(\pi^{-1}(U),\sigma|_{\pi^{-1}(U)})$
in $(S,\sigma)$ is an $\mathbb{S}^{1}$-equivariant birational diffeomorphism.
Replacing $(C,\tau)$ and $(D,h)$ by $(U,\tau|_{U})$ and $(D|_{U},h|_{U})$,
we can therefore assume that $D(q)=\mathrm{ord}_{q}(h)=0$ for every
non-real point $q$ of $C$. Now let $\tilde{D}$ be the Weil $\mathbb{Q}$-divisor
on $Q_{1,\mathbb{C}}$ defined by $\tilde{D}(c)=D(c)$ if $c\in C$
and $\tilde{D}(c)=0$ otherwise, and let $\tilde{h}=h$. Since $(C,\tau)\subset(Q_{1,\mathbb{C}},\sigma_{Q_{1}})$
is a real affine open subset with the same real locus as $(Q_{1,\mathbb{C}},\sigma_{Q_{1}})$,
$Q_{1,\mathbb{C}}\setminus C$ consists of finitely many pairs $\{q,\sigma_{Q_{1}}(q)\}$
of non-real points of $Q_{1,\mathbb{C}}$ exchanged by the real structure
$\sigma_{Q_{1}}$. For every such pair of points, we have by construction
\[
\mathrm{ord}_{q}(\tilde{h})=\mathrm{ord}_{q}(h)=\tilde{D}(q)+\tilde{D}(\sigma_{Q_{1}}(q))=\tilde{D}(q)+\sigma_{Q_{1}}^{*}(\tilde{D})(q)
\]
and similarly for $\sigma_{Q_{1}}(q)$. This implies that $(\tilde{D},\tilde{h})$
is a real DPD-pair on $(Q_{1,\mathbb{C}},\sigma_{Q_{1}})$ which is
regular since $(D,h)$ is regular. Let $(\tilde{S},\tilde{\sigma})$
be the corresponding smooth real affine surface with $\mathbb{S}^{1}$-action
and let $\tilde{\pi}:(\tilde{S},\tilde{\sigma})\rightarrow(Q_{1,\mathbb{C}},\sigma_{Q_{1}})$
be its real quotient morphism. We then have a cartesian square of
real algebraic varieties \[\xymatrix{(S,\sigma) \ar[r]^{\varphi} \ar[d]_{\pi} & (\tilde{S},\tilde{\sigma}) \ar[d]^{\tilde{\pi}} \\ (C,\tau) \ar[r] & (Q_{1,\mathbb{C}},\sigma_{Q_1})} \]
in which the top horizontal morphism $\varphi$ is an $\mathbb{S}^{1}$-equivariant
open embedding of $(S,\sigma)$ as the complement of the fibers of
$\tilde{\pi}$ over the points of $Q_{1,\mathbb{C}}\setminus C$.
Since $Q_{1,\mathbb{C}}\setminus C$ consists of pairs of non-real
points of $Q_{1,\mathbb{C}}$, the real loci of $(S,\sigma)$ and
$(\tilde{S},\tilde{\sigma})$ coincide, which implies that $\varphi:(S,\sigma)\rightarrow(\tilde{S},\tilde{\sigma})$
is a birational diffeomorphism.

By construction, $(\tilde{D},\tilde{h})$ is a regular real DPD-pair
on $(Q_{1,\mathbb{C}},\sigma_{Q_{1}})$ such that the support of $\tilde{D}$
consists of real points of $(Q_{1,\mathbb{C}},\sigma_{Q_{1}})$ and
such that $\tilde{h}$ is a regular function whose zero locus consists
of real points only. Furthermore, the image by $\tilde{\pi}:(\tilde{S},\tilde{\sigma})\rightarrow(Q_{1,\mathbb{C}},\sigma_{Q_{1}})$
of the real locus of $(\tilde{S},\tilde{\sigma})$ is equal to that
of $(Q_{1,\mathbb{C}},\sigma_{Q_{1}})$. By Lemma \ref{lem:DPD-pair-for-torus}
below, $(\tilde{S},\tilde{\sigma})$ is $\mathbb{S}^{1}$-equivariantly
isomorphic to $(Q_{1,\mathbb{C}},\sigma_{Q_{1}})\times\mathbb{S}^{1}$
on which $\mathbb{S}^{1}$ acts by translations on the second factor.
This completes the proof. 
\end{proof}
In the proof of Proposition \ref{prop:Torus-diffBir} above, we use
the following auxiliary characterization of $(Q_{1,\mathbb{C}},\sigma_{Q_{1}})\times\mathbb{S}^{1}$
up to $\mathbb{S}^{1}$-equivariant real isomorphisms: 
\begin{lem}
\label{lem:DPD-pair-for-torus}Let $(D,h)$ be a regular real DPD-pair
on $(Q_{1,\mathbb{C}},\sigma_{Q_{1}})$ such that the support of $D$
consists of real points of $(Q_{1,\mathbb{C}},\sigma_{Q_{1}})$ and
such that $h$ is a real regular function whose zero locus consists
of real points of $(Q_{1,\mathbb{C}},\sigma_{Q_{1}})$ only. Let $(S,\sigma)$
be the corresponding smooth real affine surface with $\mathbb{S}^{1}$-action
and let $\pi:(S,\sigma)\rightarrow(Q_{1,\mathbb{C}},\sigma_{Q_{1}})$
be its real quotient morphism. Then the following are equivalent:

i) The image by $\pi$ of the real locus of $(S,\sigma)$ is equal
to the real locus of $(Q_{1,\mathbb{C}},\sigma_{Q_{1}})$.

ii) The surface $(S,\sigma)$ is $\mathbb{S}^{1}$-equivariantly isomorphic
to $(Q_{1,\mathbb{C}},\sigma_{Q_{1}})\times\mathbb{S}^{1}$ on which
$\mathbb{S}^{1}$ acts by translations on the second factor. 
\end{lem}

\begin{proof}
The implication ii)$\Rightarrow$i) is clear. We now proceed to the
proof of i)$\Rightarrow$ii). For every real point $c=(c_{1},c_{2})$
of $(Q_{1,\mathbb{C}},\sigma_{Q_{1}})$, the restrictions to $Q_{1,\mathbb{C}}$
of the regular functions 
\[
F_{c}=(v-c_{2})-i(u-c_{1})\quad\textrm{and}\quad\overline{F}_{c}=(v-c_{2})+i(u-c_{2})
\]
on $\mathbb{A}_{\mathbb{C}}^{2}=\mathrm{Spec}(\mathbb{C}[u,v])$ are
regular functions $f_{c}$ and $\overline{f}_{c}$ on $Q_{1,\mathbb{C}}$
such that $\mathrm{div}(f_{c})=\mathrm{div}(\overline{f}_{c})=c$.
Furthermore, we have $\sigma_{Q_{1}}^{*}f_{c}=\overline{f}_{c}$ so
that $\mathrm{div}(f_{c}\sigma_{Q_{1}}^{*}f_{c})=2c$. Arguing as
in the proof of Lemma \ref{lem:Local-Reduction}, we obtain that $(S,\sigma)$
is $\mathbb{S}^{1}$-equivariantly isomorphic to the surface determined
by a regular real DPD-pair $(D',h')$ on $(Q_{1,\mathbb{C}},\sigma_{Q_{1}})$
such that $\mathrm{Supp}(D')$ is contained in the real locus of $(Q_{1,\mathbb{C}},\sigma_{Q_{1}})$,
$h'$ is regular and vanishes at real points only, and such that for
every real point $c$ of $(Q_{1,\mathbb{C}},\sigma_{Q_{1}})$ exactly
one of the following possibilities occurs:

a) $D'(c)=0$ and $\mathrm{ord}_{c}(h')=0$ 

b) $D'(c)=\frac{1}{2}$ and $\mathrm{ord}_{c}(h')=1$ 

c) $D'(c)=0$ and $\mathrm{ord}_{c}(h')=1$. 

Consider the restriction $h'|_{S^{1}}:S^{1}\rightarrow\mathbb{R}$
of the real regular function $h'$ to the real locus $S^{1}$ of $(Q_{1,\mathbb{C}},\sigma_{Q_{1}})$.
If $c_{0}$ is a real point of $(Q_{1,\mathbb{C}},\sigma_{Q_{1}})$
of type b) or c) then $h'|_{S^{1}}$ is a continuous function on $S^{1}$
whose sign changes at $c_{0}$. It follows that there exists a real
point $c$ of $(Q_{1,\mathbb{C}},\sigma_{Q_{1}})$ such that $D'(c)=0$
and $h'(c)<0$. But then, it follows from the proof of Theorem \ref{prop:Real-Except-Orb-1}
1) that $(\pi^{-1}(c),\sigma|_{\pi^{-1}(c)})$ is $\mathbb{S}^{1}$-equivariantly
isomorphic to the nontrivial $\mathbb{S}^{1}$-torsor $\hat{\mathbb{S}}^{1}$
which has empty real locus. This is impossible since by hypothesis
the real locus of $(S,\sigma)$ surjects onto that of $(Q_{1,\mathbb{C}},\sigma_{Q_{1}})$.
Thus $D'(c)=\mathrm{ord}_{c}(h')=0$ for every real point $c$ of
$(Q_{1,\mathbb{C}},\sigma_{Q_{1}})$. This implies in turn $D'$ has
empty support and that $h'$ is a nowhere vanishing real regular function
on $(Q_{1,\mathbb{C}},\sigma_{Q_{1}})$. It follows that $h'$ is
constant, with positive value $\lambda\in\mathbb{R}_{+}^{*}$ at every
point since the real locus of $(S,\sigma)$ surjects onto that of
$(Q_{1,\mathbb{C}},\sigma_{Q_{1}})$. Writing $\lambda=\alpha\tau^{*}\alpha=\alpha^{2}$
for some real number $\alpha$, we deduce from Theorem \ref{thm:MainThm}
1) that the surface $(S,\sigma)$ is $\mathbb{S}^{1}$-equivariantly
isomorphic to that defined by the real DPD-pair $(D',1)=(0,1)$ on
$(Q_{1,\mathbb{C}},\sigma_{Q,1})$. By subsection \ref{subsec:Torus-Model},
the latter is $\mathbb{S}^{1}$-equivariantly isomorphic to $(Q_{1,\mathbb{C}},\sigma_{Q_{1}})\times\mathbb{S}^{1}$
on which $\mathbb{S}^{1}$ acts by translations on the second factor. 
\end{proof}

\subsubsection{\label{subsec:Second-case}Second case: $(C,\tau)$ is a real affine
open subset of the real affine line }
\begin{prop}
Let $(S,\sigma)$ be a smooth rational real affine surface with an
effective $\mathbb{S}^{1}$-action and whose real locus is a connected
compact surface without boundary. Let $\pi:(S,\sigma)\rightarrow(C,\tau)$
be its real quotient morphism and assume that $(C,\tau)$ is a real
affine open subset of $(\mathbb{A}_{\mathbb{C}}^{1},\sigma_{\mathbb{A}_{\mathbb{R}}^{1}})$.
Then $(S,\sigma)$ is $\mathbb{S}^{1}$-equivariantly birationally
diffeomorphic to one of the affine surfaces $(S_{2},\sigma_{2})$,
$(S_{3},\sigma_{3})$ and $(S_{4},\sigma_{4})$ in Table \ref{RatMods}.
\end{prop}

\begin{proof}
Let $\mathbb{A}_{\mathbb{C}}^{1}=\mathrm{Spec}(\mathbb{C}[z])$ and
let $(D,h)$ be a regular real DPD-pair on $(C,\tau)\subseteq(\mathbb{A}_{\mathbb{C}}^{1},\sigma_{\mathbb{A}_{\mathbb{R}}^{1}})$
corresponding to $(S,\sigma)$. The image by $\pi$ of the real locus
of $(S,\sigma)$ is closed interval $J$ of the real locus $\mathbb{R}$
of $(\mathbb{A}_{\mathbb{C}}^{1},\sigma_{\mathbb{A}_{\mathbb{R}}^{1}})$.
By Theorem \ref{thm:MainThm} 1), $(S,\sigma)$ is $\mathbb{S}^{1}$-equivariantly
isomorphic to the surface determined by a real DPD-pair of the form
\[
(D+\mathrm{div}(f|_{C}),(f\tau^{*}f)|_{C}h)
\]
on $(C,\tau)$, where $f$ is any element of $\mathbb{C}(z)$. By
choosing $f\in\mathbb{C}(z)$ suitably, we can assume from the very
beginning that $h\in\mathbb{C}(z)$ is a real polynomial whose zero
locus on $\mathbb{A}_{\mathbb{C}}^{1}$ is contained in $J$. Then
arguing as in the proof of Proposition \ref{prop:Torus-diffBir},
we get that $(S,\sigma)$ is $\mathbb{S}^{1}$-equivariantly birationally
diffeomorphic to the smooth real affine surface $(\pi^{-1}(U),\sigma|_{\pi^{-1}(U)})$
determined by the real DPD pair $(D|_{U},h|_{U})$ on the open complement
$(U,\tau|_{U})$ in $C$ of the set of non-real points $q$ of $(C,\tau)$
such that either $D(q)\neq0$ or $D(\tau(q))\neq0$. Then $(\pi^{-1}(U),\sigma|_{\pi^{-1}(U)})$
is in turn $\mathbb{S}^{1}$-equivariantly birationally diffeomorphic
to the smooth real affine surface $(\tilde{S},\tilde{\sigma})$ determined
by the regular real DPD-pair $(\tilde{D},\tilde{h})=(\tilde{D},h)$
on $(\mathbb{A}_{\mathbb{C}}^{1},\sigma_{\mathbb{A}_{\mathbb{R}}^{1}})$,
where $\tilde{D}$ is the Weil $\mathbb{Q}$-divisor defined by $\tilde{D}(c)=D(c)$
if $c\in J$ and $\tilde{D}(c)=0$ otherwise. 

By composing the real quotient morphism $\tilde{\pi}:(\tilde{S},\tilde{\sigma})\rightarrow(\mathbb{A}_{\mathbb{C}}^{1},\sigma_{\mathbb{A}_{\mathbb{R}}^{1}})$
by a real automorphism of $(\mathbb{A}_{\mathbb{C}}^{1},\sigma_{\mathbb{A}_{\mathbb{R}}^{1}})$,
we can assume without loss of generality that $J$ is equal to the
interval $[-1,1]$. For every real point $c$ of $J$, $f_{c}=z-c$
is a real regular function $f_{c}$ on $(\mathbb{A}_{\mathbb{C}}^{1},\sigma_{\mathbb{A}_{\mathbb{R}}^{1}})$
such that $c=\mathrm{div}(f_{c})$. Arguing as in the proof of Lemma
\ref{lem:Local-Reduction}, we obtain that $(\tilde{S},\tilde{\sigma})$
is $\mathbb{S}^{1}$-equivariantly isomorphic to the surface $(\hat{S},\hat{\sigma})$
determined by a real DPD-pair $(\hat{D},\hat{h})$ on $(\mathbb{A}_{\mathbb{C}}^{1},\sigma_{\mathbb{A}_{\mathbb{R}}^{1}})$
such that $\mathrm{Supp}(\hat{D})$ is contained in $J$, $\hat{h}$
is a regular function whose zero locus is contained in $J$ and such
that that for every $c\in J$ exactly one of the following possibilities
occurs:

a) $\hat{D}(c)=0$ and $\mathrm{ord}_{c}(\hat{h})=0$ 

b) $\hat{D}(c)=\frac{1}{2}$ and $\mathrm{ord}_{c}(\hat{h})=1$ 

c) $\hat{D}(c)=0$ and $\mathrm{ord}_{c}(\hat{h})=1$. 

By composing the real  quotient morphism $\hat{\pi}:(\hat{S},\hat{\sigma})\rightarrow(\mathbb{A}_{\mathbb{C}}^{1},\sigma_{\mathbb{A}_{\mathbb{R}}^{1}})$
by the real automorphism $z\mapsto-z$ of $(\mathbb{A}_{\mathbb{C}}^{1},\sigma_{\mathbb{A}_{\mathbb{R}}^{1}})$,
we can further assume without loss of generality that $\hat{D}(-1)\geq\hat{D}(1)$.
By Theorem \ref{prop:Real-Except-Orb-1}, for a real point $c$ of
$(\mathbb{A}_{\mathbb{C}}^{1},\sigma_{\mathbb{A}_{\mathbb{R}}^{1}})$,
the real locus of $(\hat{\pi}^{-1}(c),\hat{\sigma}|_{\hat{\pi}^{-1}(c)})$
is empty if and only if $\hat{D}(c)=\mathrm{ord}_{c}(\hat{h})=0$
and $\hat{h}(c)<0$. It follows that $\hat{D}(c)=\mathrm{ord}_{c}(\hat{h})=0$
and $\hat{h}(c)<0$ for every real point $c$ of $(\mathbb{A}_{\mathbb{C}}^{1},\sigma_{\mathbb{A}_{\mathbb{R}}^{1}})$
outside of $J$. On the other hand, the real locus of $(\hat{\pi}^{-1}(c),\hat{\sigma}|_{\hat{\pi}^{-1}(c)})$
being nonempty for every real point $c\in J$ by assumption, we have
$\hat{h}(c)\geq0$ for every $c\in J$. It follows that $\hat{h}\in\mathbb{R}[z]\subset\mathbb{C}[z]$
is a nonzero real polynomial with only simple real roots, whose restriction
to the real locus $\mathbb{R}$ of $(\mathbb{A}_{\mathbb{C}}^{1},\sigma_{\mathbb{A}_{\mathbb{R}}^{1}})$
is negative outside $J$ and non-negative on $J$. This implies that
$\hat{h}=\lambda(1-z^{2})$ for some $\lambda\in\mathbb{R}_{+}^{*}$
which can be further chosen equal to $1$ by Theorem \ref{thm:MainThm}
1). It follows in turn that $\hat{D}(c)=0$ for every real point $c$
of $(\mathbb{A}_{\mathbb{C}}^{1},\sigma_{\mathbb{A}_{\mathbb{R}}^{1}})$
other than $-1$ and $1$, and since $\mathrm{ord}_{\pm1}\hat{h}=1$
and $\hat{D}(-1)\geq\hat{D}(1)$, the only remaining possibilities
are the following:

(i) $\hat{D}(-1)=\hat{D}(1)=0$

(ii) $\hat{D}(-1)=\frac{1}{2}$ and $\hat{D}(1)=0$ 

(iii) $\hat{D}(-1)=\hat{D}(1)=\frac{1}{2}.$

These pairs $(\hat{D},\hat{h})$ correspond respectively to the model
$(S_{2},\sigma_{2})$ of $S^{2}$, $(S_{3},\sigma_{3})$ of $\mathbb{RP}^{2}$
and $(S_{4},\sigma_{4})$ of $K$ in Table \ref{RatMods}. This completes
the proof.  
\end{proof}
\bibliographystyle{amsplain}

\end{document}